\documentclass[english,13pt]{article}
\usepackage[T1]{fontenc}
\usepackage[latin1]{inputenc}
\usepackage{geometry}
\usepackage{float}
\usepackage{mathrsfs}
\usepackage{amsmath}
\usepackage{amssymb}
\usepackage{graphicx}
\usepackage{hyperref}
\hypersetup{
    colorlinks=true,
    linkcolor=blue,
    filecolor=magenta,      
    urlcolor=cyan,
    citecolor = red,
}

\makeatletter

\usepackage{algorithm,algpseudocode}

\usepackage{amsthm}

\usepackage{mathrsfs}

\usepackage{amsfonts}

\usepackage{epsfig}

\usepackage{bm}

\usepackage{mathrsfs}

\usepackage{enumerate}

\@ifundefined{definecolor}{\@ifundefined{definecolor}
 {\@ifundefined{definecolor}
 {\usepackage{color}}{}
}{}
}{}

\usepackage{subfig}\usepackage[all]{xy}

\newtheorem{theorem}{Theorem}[section]
\newtheorem{lem}{Lemma}[section]
\newtheorem{rem}{Remark}[section]

\newcounter{hypA}
\newenvironment{hypA}{\refstepcounter{hypA}\begin{itemize}
  \item[({\bf A\arabic{hypA}})]}{\end{itemize}}
\newcounter{hypB}

\newcounter{hypD}

\newcounter{hypW}
\newenvironment{hypW}{\refstepcounter{hypW}\begin{itemize}
  \item[({\bf W\arabic{hypW}})]}{\end{itemize}}

\usepackage{babel}\date{}

\usepackage{babel}

\makeatother

\usepackage{babel}

\begin{document}

%+Title
\begin{center}

{\Large \textbf{Multilevel Particle Filters for Partially Observed McKean-Vlasov Stochastic Differential Equations}}

\vspace{0.5cm}

BY  ELSIDDIG AWADELKARIM $^{1}$ \& AJAY JASRA$^{2}$

{\footnotesize $^{1}$Applied Mathematics and Computational Science Program,  Computer, Electrical and Mathematical Sciences and Engineering Division, King Abdullah University of Science and Technology, Thuwal, 23955-6900, KSA.}\\
{\footnotesize $^{2}$School of Data Science, The Chinese University of Hong Kong, Shenzhen, CN.}
{\footnotesize E-Mail:\,} \texttt{\emph{\footnotesize elsiddigawadelkarim.elsiddig@kaust.edu.sa; ajayjasra@cuhk.edu.cn
}}

\end{center}

\begin{abstract}
In this paper we consider the filtering problem associated to partially observed McKean-Vlasov stochastic differential equations (SDEs). The model consists of data that are observed at regular and discrete times and the objective is to compute the conditional expectation of (functionals) of the solutions of the SDE at the current time.
This problem, even the ordinary SDE case is challenging and requires numerical approximations. Based upon the ideas in \cite{ml_is_mv,dosreis} we develop a new particle filter (PF) and multilevel particle filter (MLPF) to approximate the afore-mentioned expectations. We prove under assumptions that, for $\epsilon>0$, to obtain a mean square error of $\mathcal{O}(\epsilon^2)$ the PF has a cost per-observation time of $\mathcal{O}(\epsilon^{-5})$ and the MLPF costs $\mathcal{O}(\epsilon^{-4})$ (best case) or $\mathcal{O}(\epsilon^{-4}\log(\epsilon)^2)$ (worst case). Our theoretical results are supported by numerical experiments.
\end{abstract}

\section{Introduction}

We consider the stochastic differential equation (SDE) with $X_0=x_0\in\mathbb{R}^d$ fixed:
\begin{equation}\label{eq:sde}
dX_t = a\left(X_t,\overline{\xi}_1(X_t,\mu_t)\right)dt + b\left(X_t,\overline{\xi}_2(X_t,\mu_t)\right)dW_t
\end{equation}
where for $j\in\{1,2\}$
$$
\overline{\xi}_j(X_t,\mu_t)  =  \int_{\mathbb{R}^d}\xi_j(X_t,x)\mu_t(dx)
$$
$\{W_t\}_{t\geq 0}$ is a standard $d-$dimensional Brownian motion,  $\xi_j:\mathbb{R}^{2d}\rightarrow\mathbb{R}$, $a:\mathbb{R}^d\times\mathbb{R}\rightarrow\mathbb{R}^d$, $b:\mathbb{R}^d\times\mathbb{R}\rightarrow\mathbb{R}^{d\times d}$ and $\mu_t$ is the law of $X_t$.  This is the collection McKean-Vlasov SDEs
and finds a wide variety of real applications such as biology and finance \cite{bush,dob}, in addition they have been shown to be equivalent representation of solutions of classes of nonlinear partial differential equations \cite{barbu},  and one can expect even more applications in the future. 
The exact problem that will be addressed in this article is where the process defined in \eqref{eq:sde} is only observed with noise at regular and discrete times. The objective is then to compute the conditional expectation of $X_t$ (or functionals thereof) given all the observations recursively at every observation time, which is a filtering problem; see e.g.~\cite{bain,cappe} for an introduction.

The filtering of continuous-time processes is very challenging. First, even if the transition densities of the process exist and are analytically tractable,  filtering must normally be perfomed by using numerical methods of which we focus upon particle filters (PF); see e.g.~\cite{bain,cappe,delmoral}. Second,  the afore-mentioned transition densities are seldom available and hence one normally has to resort to a time discretization of the process; see for instance \cite{mlpf}. In the case of regular diffusions (i.e.~where the coefficients do no depend on $\mu_t$), there is already a substantial literature on using PFs for solving the filtering problem; see e.g.~\cite{delmoral1,fearn}.  In addition to these methods PFs have been extended to leverage on the popular multilevel Monte Carlo (MLMC) method \cite{giles,giles1,hein}.
These multilevel PFs (MLPFs) \cite{mlpf,mlpf1,anti_mlpf} can reduce the cost to compute the filtering expectation for a given mean square error (MSE).

To the best of our knowledge,  there is little work in the context of filtering McKean-Valasov SDEs. However, when
trying to compute expectations with respect to (w.r.t.) the laws of such processes, several methods have been proposed. This includes the initial work of \cite{basic_method} that relies on time discretization and simulation coupled with solving backward Kolmogorov equations; this can be prohibitive if $d$ is even moderate.  Several recent extensions based upon using the afore-mentioned simulation and designing change of measures for importance sampling have been considered in \cite{dosreis} and then combining with MLMC \cite{ml_is_mv} or multi-index MC (see \cite{haji}) \cite{mi_is_mv}.  An iterative method, combined with MLMC was also proposed in \cite{szp}.  Despite these high quality and interesting contributions, none of these methods are designed for the filtering problem and would likely have several deficiencies.  For instance,  if the change of measure methods would be adopted it is likely that they would suffer from the well-known weight-degeneracy problem (e.g.~\cite{cappe,delmoral}). That is,  the estimates of the filter would suffer an exponential increase of variance w.r.t.~to observation times. The approach of \cite{szp} is iterative and thus is not appropriate for the filtering problem as one seeks fast (non-iterative) methods to assimilate the data.

In this paper we develop new PF and MLPF methodology to estimate the filter, associated to partially observed Mckean-Vlasov SDEs, recursively in time. This requires some care as, unlike the case of filtering ordinary diffusions, even under a time-discretization, one cannot simulate the SDE. This is because the law of the discretized process is not known and must be estimated. We show how the idea adopted in \cite{ml_is_mv,dosreis} can be leveraged. That is, to use the particle system of \cite{basic_method} to be fed into a `classical' PF and this resulting algorithm can be used to recursively estimate the filter. We then show how this idea can be extended to the multilevel context by combining with ideas from \cite{mlpf,  mlpf_cts}; see \cite{ml_rev} for a review and \cite{ubpf} for related methodology.
We prove, under assumptions, that for, $\epsilon>0$ to achieve a MSE, associated to estimating the filter, of $\mathcal{O}(\epsilon^2)$ per-observation time:
\begin{itemize}
\item{the cost for the PF is $\mathcal{O}(\epsilon^{-5})$}
\item{the cost for the MLPF is $\mathcal{O}(\epsilon^{-4})$ (best case) or $\mathcal{O}(\epsilon^{-4}\log(\epsilon)^2)$ (worst case).}
\end{itemize}
The best case corresponds to the case where the diffusion coefficient $b$ in \eqref{eq:sde} is constant and the other case is for the non-constant case - of course only when the SDE follows our assumptions.
The costs above are higher than one might expect for ordinary Monte Carlo based methods (of which PFs and MLPFs are). These increased costs are primarily associated to having to approximate the laws $\mu_t$ (and their time-discretized versions) in the evolution of the SDE.  There are possible ways to deal with this, but it is not investigated in this paper. The MSE to cost results are verified in several numerical examples.

%The objective is to compute, for $g:\mathbb{R}^d\rightarrow\mathbb{R}$:
%$$
%\mu_T(g) := \mathbb{E}[g(X_T)] = \int_{\mathbb{R}^d}g(x)\mu_T(dx).
%$$

%For convenience of notation, we suppose that $T\in\mathbb{N}$ and denote by $P_{\mu_t,t}(x_{t-1},dx_t)$ the conditional law of $x_t$ given $\mathscr{F}_{t-1}$ (the natual filtration of the process), for $t\in\{1,\dots,T\}$.
%As was realized in \cite{smc_diff} a method which can greatly enhance over exact simulation of $X_T$ (in terms of variance) is to consider the sequence of target measures:
%$$
%\pi_t\left(dx_{1:t}\right) \propto g(x_t) \prod_{k=1}^t P_{\mu_k,k}(x_{k-1},dx_k)
%$$
%The normalizing constant of $\pi_T$ is exactly $\mu_T(g)$. Such a sequence of probabilities, and their normalizing constants, can be approximated by using a particle filter, with proposal (at time $t\in\{1,\dots,T\}$) $P_{\mu_t,t}(x_{t-1},dx_t)$ and weight:
%$$
%G_{t-1}(x_{t-1:t}) = \frac{g(x_t)}{g(x_{t-1})}
%$$
%with $g(x_0):=1$.  It is simple to estimate the normalizing constant of $\pi_T$ using a particle filter and we do not detail that here.  
%
%There are of course, two obvious issues with the method:
%\begin{itemize}
%\item{The law $\mu_t$ is unknown.}
%\item{Even if $\mu_t$ is known, one can seldom sample from the exact transition law.}
%\end{itemize}
%We will construct a method that can deal with this issue.

This article is structured as follows. In Section \ref{sec:approach} we give our approach to filtering partially observed McKean-Vlasov SDEs.  In Section \ref{sec:theory} our theoretical results are presented.  Section \ref{sec:numerics} details our numerical results associated to our methodology.  Our proofs can be found in the Appendix.

\section{Approach}\label{sec:approach}

\subsection{State-Space Model}

We denote by $P_{\mu_{t-1},t}(x_{t-1},dx_t)$ the conditional law of $X_t$ (as given in \eqref{eq:sde}) given $\mathscr{F}_{t-1}$ (the natual filtration of the process), for $t\geq 1$; that is, the transition kernel over unit time. We consider a discrete time observation
process $Y_1,Y_2,\dots$,  $Y_t\in\mathsf{Y}$, that are assumed, for notational convenience, to be observed at unit times. Conditional on the position $X_t$, $t\in\mathbb{N}$ of \eqref{eq:sde}, the random variable $Y_t$ is assumed to be independent of all other random variables, with a bounded and positive probability density $G(x_t,y_t)$.

Let $\varphi:\mathbb{R}^d\rightarrow\mathbb{R}$ be a bounded and measurable function (write such a collection of functions $\mathcal{B}_b(\mathbb{R}^d)$), then we define the filtering expectation for $t\in\mathbb{N}$ as:
$$
\pi_t(\varphi) := \frac{\int_{\mathbb{R}^{dt}}\varphi(x_t)\left\{\prod_{p=1}^t G(x_p,y_p)\right\}\prod_{p=1}^tP_{\mu_{p-1},p}(x_{p-1},dx_p)}{\int_{\mathbb{R}^{dt}}\left\{\prod_{p=1}^t G(x_p,y_p)\right\}\prod_{p=1}^tP_{\mu_{p-1},p}(x_{p-1},dx_p)}
$$
where $\mu_0=\delta_{\{x_0\}}$ (Dirac measure on the set $\{x_0\}$). We remark that one does not need $\varphi$
and each of the $G(\cdot,y_t)$ to be bounded, but it will simplify the resulting exposition to do so.

In most cases of practical interest, $\mu_t$ and the dynamics $P_{\mu_{t-1},t}(x_{t-1},dx_t)$ are difficult to work with.  For instance the transition kernel cannot be simulated in many problems. As a first step to approximate the filter $\pi_t$, we introduce a time-discretization over a regular grid of spacing $\Delta_l=2^{-l}$, $l\in\mathbb{N}_0$. The significance of subscript $l$ is that it will denote the level of discretization - as $l$ grows so will the accuracy of the time discretization. We will use the Euler-Maruyama method associated to \eqref{eq:sde} and denote the law at any time $t\in\{0,\Delta_l,2\Delta_l,\dots\}$ as $\mu_t^l$. That is, we now consider the approximation for $k\in\mathbb{N}_0$:
\begin{eqnarray}
X_{k\Delta_l} & = & X_{(k-1)\Delta_l} + a\left(X_{(k-1)\Delta_l},\overline{\xi}_1(X_{(k-1)\Delta_l},\mu_{(k-1)\Delta_l}^l)\right) +\nonumber \\ & &b\left(X_{(k-1)\Delta_l},\overline{\xi}_2(X_{(k-1)\Delta_l},\mu_{(k-1)\Delta_l}^l)\right)\left[W_{k\Delta_l} - W_{(k-1)\Delta_l}\right]\label{eq:euler}
\end{eqnarray}
where $X_0=x_0$ and $\mu_0^l=\delta_{\{x_0\}}$.  Associated to \eqref{eq:euler}, we denote by $P_{\mu_{t-1},t}^l(x_{t-1},dx_t)$ the conditional law of $X_t$, $t\in\mathbb{N}$,  given $\mathscr{F}_{t-1}$ for $t\geq 1$; that is, the transition kernel over unit time induced by \eqref{eq:euler}.  It should be remarked that in many cases, 
\eqref{eq:euler} cannot be simuated exactly as the expectations associated to $\mu_{(k-1)\Delta_l}^l$ cannot be computed even if one knows $\mu_{(k-1)\Delta_l}^l$, which is again unlikely.

Our objective now is to compute an approximation of the time discretized filtering expectation for $(t,l,\varphi)\in\mathbb{N}\times\mathbb{N}_0\times\mathcal{B}_b(\mathbb{R}^d)$:
$$
\pi_t^l(\varphi) := \frac{\int_{\mathbb{R}^{dt}}\varphi(x_t)\left\{\prod_{p=1}^t G(x_p,y_p)\right\}\prod_{p=1}^tP_{\mu_{p-1}^l,p}^l(x_{p-1},dx_p)}{\int_{\mathbb{R}^{dt}}\left\{\prod_{p=1}^t G(x_p,y_p)\right\}\prod_{p=1}^tP_{\mu_{p-1}^l,p}^l(x_{p-1},dx_p)}.
$$

\subsection{Approximating the Law}

We now consider a method that will be used to provide a Monte Carlo based approximation of the law
$\mu_t^l$. The approach we present is a simple discretized method in Algorithm \ref{alg:basic_method} from \cite{basic_method}.   In Algorithm \ref{alg:basic_method}, the notation $\mathcal{N}_d(\kappa,\Sigma)$ denotes the $d-$dimensional Gaussian distribution with mean $\kappa$ and covariance matrix $\Sigma$. $I_d$ is the $d\times d$
identity matrix and $\stackrel{\textrm{ind}}{\sim}$ denotes independently distributed as.

Algorithm \ref{alg:basic_method} can be used to approximate expectations w.r.t.~$\mu_t^l$
and indeed on the grid in-between time $t-1$ and $t$.   Algorithm \ref{alg:basic_method} is given in the form that we need it later on in the article.

\begin{algorithm}[h]
\begin{enumerate}
\item{Input $l\in\mathbb{N}_0$ the level of discretization, $N\in\mathbb{N}$ the number of particles,  $t\in\{1,\dots,T\}$. If $t=1$ set $\mu_0^N(dx)=\delta_{\{x_0\}}(dx)$ otherwise input an empirical measure $\mu_{t-1}^N(dx)=\tfrac{1}{N}\sum_{i=1}^N\delta_{\{X_{t-1}^i\}}(dx)$. Set $k=1$.}
\item{For $i\in\{1,\dots,N\}$ generate:
\begin{align*}
X_{t-1+k\Delta_l}^i & =  X_{t-1+(k-1)\Delta_l}^i + a\left(X_{t-1+(k-1)\Delta_l}^i,\overline{\xi}_1(X_{t-1+(k-1)\Delta_l}^i,\mu_{t-1+(k-1)\Delta_l}^N)\right) + \\ & b\left(X_{t-1+(k-1)\Delta_l}^i,\overline{\xi}_2(X_{t-1+(k-1)\Delta_l}^i,\mu_{t-1+(k-1)\Delta_l}^N)\right)\left[W_{t-1+k\Delta_l}^i - W_{t-1+(k-1)\Delta_l}^i\right]
\end{align*}
where
\begin{eqnarray*}
\overline{\xi}_m(X_{t-1+(k-1)\Delta_l}^i,\mu_{t-1+(k-1)\Delta_l}^N) & = & \frac{1}{N}\sum_{j=1}^N \xi_m(X_{t-1+(k-1)\Delta_l}^i,X_{t-1+(k-1)\Delta_l}^j)\quad m\in\{1,2\}\\
\mu_{t-1+(k-1)\Delta_l}^N(dx) & = & \frac{1}{N}\sum_{j=1}^N\delta_{\{X_{t-1+(k-1)\Delta_l}^j\}}(dx) \\
\left[W_{t-1+k\Delta_l}^i - W_{t-1+(k-1)\Delta_l}^i\right] & \stackrel{\textrm{ind}}{\sim} & \mathcal{N}_{d}(0,\Delta_l I_d).
\end{eqnarray*}
Set $k=k+1$, if $k=\Delta_l^{-1}+1$ go to step 3.~otherwise go to the start of step 2..}
\item{Output all the required laws $\mu_{t-1+\Delta_l}^N,\dots,\mu_{t}^N$.}
\end{enumerate}
\caption{Approximating the Laws when starting with a particle approximation at time $t-1$, $t\in\{1,\dots,T\}$.}
\label{alg:basic_method}
\end{algorithm}

\subsection{Particle Filter}

An approach for the recursive approximation of $\pi_t^l(\varphi)$ is
 presented in Algorithm \ref{alg:pf} and is run to a terminal time $T$ - this need not be the case.
At this stage, several remarks are of interest:
\begin{enumerate}
\item{The cost of Algorithm \ref{alg:pf} is $\mathcal{O}(\Delta_l^{-1}M(M+N))$ per observation time (or time step).}
\item{In step 2.~it is critical that this process be independent of the particles that are generated in the particle filter.  The empirical laws produced here are generated independently of anything that is done later on.
If one used the empirical laws that were produced by the samples from the particle filter, then the approximation would be biased even if $N,M\rightarrow\infty$ and $l\rightarrow\infty$. This is because these particles used in the particle filter have been designed to approximate expectations w.r.t.~$\pi_t$ and not the law of the McKean-Vlasov SDE itself. Thus, here the motivation of a driving approximation of the SDE marginal law, as was the case in \cite{ml_is_mv,dosreis},  is one of necessity.} 
\item{The variance w.r.t.~the time parameter should be controlled and can be uniform in time.
This may not be the case for the methods in \cite{ml_is_mv}.}
\item{Improved discretization methods can be used, when available.}
\item{We will prove, later on, that as $N,M\rightarrow\infty$ we have that $\pi_{t}^{l,N,M}(\varphi)\rightarrow_{\mathbb{P}}\pi_{t}^{l}(\varphi)$, where $\rightarrow_{\mathbb{P}}$ denotes convergence in probability and $\varphi$ is in a particular class of real-valued functions to be stated.}
\end{enumerate}

\begin{algorithm}[H]
\begin{enumerate}
\item{Input $l\in\mathbb{N}_0$ the level of discretization, $N\in\mathbb{N}$ the number of particles, $M$ the number of particles to compute laws of the process,  $T\in\mathbb{N}$ the final time and 
$X_0^{i}=x_0$, $i\in\{1,\dots,N\}$.  Set $t=1$.}% and $\mu_{t-1}^{N,M}(g) = 1$.}
\item{Run Algorithm \ref{alg:basic_method} at time $t$ with $M$ particles and discretization level $l$ and, if needed and possible, using the input empirical measure that has been computed at the last call of Algorithm \ref{alg:basic_method}.}
\item{For $(i,k)\in\{1,\dots,N\}\times\{1,\dots,\Delta_l^{-1}\}$ generate:
\begin{align*}
X_{t-1+k\Delta_l}^i & =  X_{t-1+(k-1)\Delta_l}^i + a\left(X_{t-1+(k-1)\Delta_l}^i,\overline{\xi}_1(X_{t-1+(k-1)\Delta_l}^i,\mu_{t-1+(k-1)\Delta_l}^M)\right) + \\ &b\left(X_{t-1+(k-1)\Delta_l}^i,\overline{\xi}_2(X_{t-1+(k-1)\Delta_l}^i,\mu_{t-1+(k-1)\Delta_l}^M)\right)\left[W_{t-1+k\Delta_l}^i - W_{t-1+(k-1)\Delta_l}^i\right]
\end{align*}
where the laws $\mu_{t-1+\Delta_l}^M,\dots,\mu_{t-1-\Delta_l}^M$ have been computed in step 2..}
\item{For $i\in\{1,\dots,N\}$ compute the weights
$$
V^i_t = \frac{G(X_{t}^i,y_t)}{\sum_{j=1}^N G(X_{t}^i,y_t)}
$$
and $\pi_{t}^{l,N,M}(\varphi):= \sum_{i=1}^NV^i_t\varphi(X_{t}^i)$.  For $i\in\{1,\dots,N\}$ sample an index
$j^i\in\{1,\dots,N\}$ using the probability mass function $V^1_t,\dots,V^N_t$ and set 
$\check{X}_t^i=X_t^{j^i}$.
For $i\in\{1,\dots,N\}$, set $X_t^i=\check{X}_t^i$.
Set $t=t+1$ and if $t=T+1$ go to step 4.~otherwise go to step 2..}
\item{Return the estimators $\pi_{1}^{l,N,M}(\varphi),\dots,\pi_{T}^{l,N,M}(\varphi)$.}
\end{enumerate}
\caption{Particle Filter for approximating $\pi_{t}^l(\varphi)$.}
\label{alg:pf}
\end{algorithm}

\subsection{Multilevel Particle Filter}

We now show how the PF can be extended into the MLMC framework,  along the lines of the work in \cite{mlpf}.
The purpose of this is to reduce the computing cost in terms of an MSE, which is to be defined in Section \ref{sec:theory}.  The basic idea is to approximate the identity for $(L,t,\varphi)\in\mathbb{N}^2\times\mathcal{B}_b(\mathbb{R}^d)$:
$$
\pi_t^{L}(\varphi) = \pi_t^0(\varphi) + \sum_{l=1}^L \left\{\pi_t^l(\varphi)-\pi_t^{l-1}(\varphi)\right\}.
$$
The terms on the right hand side (R.H.S.) of the above equation are then approximated as follows. The term
$\pi_t^0(\varphi)$ can be dealt with using the PF of Algorithm \ref{alg:pf}. Then independently of using the PF and for each summand, we run what is called a coupled PF (CPF). This is a type of PF which allows one to approximate 
$\pi_t^l(\varphi)-\pi_t^{l-1}(\varphi)$ by producing a coupling of the sampling step (Algorithm \ref{alg:pf} step 3.)
and a coupling of the resampling step (Algorithm \ref{alg:pf} step 4.).  
In Algorithm \ref{alg:pf} we also use Algorithm \ref{alg:basic_method} which needs to be coupled as well.
This coupling and also for the sampling step of a PF is achieved using the well-known synchronous coupling for diffusions (and the associated Euler approximation) and is given in our coupling of Algorithm \ref{alg:basic_method}, which can be found in Algorithm \ref{alg:basic_method_coup}.
The coupling of the resampling step relies on simulating a maximal coupling and is given in
Algorithm \ref{alg:max_couple}.  As noted in \cite{coup_clt}, there is nothing optimal about using this approach in a CPF,  but it is rather popular in the literature and is of linear complexity in terms of the cardinality of the state-space.

Given all the ideas already stated,  we are now in a position to give the CPF algorithm which is Algorithm \ref{alg:cpf}.  The cost of the CPF is $\mathcal{O}(\Delta_l^{-1}M(M+N))$ per observation time.
We can also describe the MLPF method that we will use for the approximation of $\pi_t^{L}(\varphi)$.
We assume that we have $L\in\mathbb{N}$ given and for each level $l$, the number of samples $N_l$ used at each level of the application of Algorithm \ref{alg:cpf}, with the number of samples used for the call of 
Algorithm \ref{alg:basic_method_coup} within Algorithm \ref{alg:cpf} equal to $M_l$ also; we shall describe how to choose these numbers later on.  Then the procedure would be as follows:
\begin{enumerate}
\item{Run Algorithm \ref{alg:pf} with $N_0$ samples and $M_0$ samples for each call of Algorithm \ref{alg:basic_method}.}
\item{Independently for each $l\in\{1,\dots,L\}$, run Algorithm \ref{alg:cpf} with $N_l$ samples and $M_l$ samples for each call of Algorithm \ref{alg:basic_method_coup}.}
\end{enumerate}
Then our MLPF estimator of $\pi_t^L(\varphi)$ is for $(L,t,\varphi)\in\mathbb{N}^2\times\mathcal{B}_b(\mathbb{R}^d)$:
$$
\widehat{\pi_t^L(\varphi)} = \pi_t^{0,N_0,M_0}(\varphi) + \sum_{l=1}^L \left\{\pi_{t}^{l,N_l,M_l}(\varphi) - \pi_{t}^{l-1,N_l,M_l}(\varphi)\right\}
$$
where $\pi_t^{0,N_0,M_0}(\varphi)$ has been computed using Algorithm \ref{alg:pf} and each of the
$\pi_{t}^{l,N_l,M_l}(\varphi) - \pi_{t}^{l-1,N_l,M_l}(\varphi)$ are produced by using Algorithm \ref{alg:cpf}.
The cost of computing this estimator, per time step, is 
$$
\mathcal{O}\left(\sum_{l=0}^L\Delta_{l}^{-1}M_l(M_l+N_l)\right).
$$
In the next section we shall establish how to choose $L$, $(N_0,M_0),\dots,(N_L,M_L)$ so as to obtain a particular MSE and minimize the cost.

\begin{algorithm}
\begin{enumerate}
\item{Input $l\in\mathbb{N}_0$ the level of discretization, $N\in\mathbb{N}$ the number of particles,  $t\in\{1,\dots,T\}$. If $t=1$ set $\mu_0^{l,N}(dx)=\widetilde{\mu}_0^{l-1,N}(dx)=\delta_{\{x_0\}}(dx)$ otherwise input a pair of empirical measures $\mu_{t-1}^{l,N}(dx)=\tfrac{1}{N}\sum_{i=1}^N\delta_{\{X_{t-1}^{l,i}\}}(dx)$, 
$\widetilde{\mu}_{t-1}^{l-1,N}(dx)=\tfrac{1}{N}\sum_{i=1}^N\delta_{\{\widetilde{X}_{t-1}^{l-1,i}\}}(dx)$.  Set $k=1$.}
\item{For $i\in\{1,\dots,N\}$ generate:
\begin{align*}
X_{t-1+k\Delta_l}^{l,i} & =  X_{t-1+(k-1)\Delta_l}^{l,i} + a\left(X_{t-1+(k-1)\Delta_l}^{l,i},\overline{\xi}_1(X_{t-1+(k-1)\Delta_l}^{l,i},\mu_{t-1+(k-1)\Delta_l}^{l,N})\right) + \\ & b\left(X_{t-1+(k-1)\Delta_l}^{l,i},\overline{\xi}_2(X_{t-1+(k-1)\Delta_l}^{l,i},\mu_{t-1+(k-1)\Delta_l}^{l,N})\right)\left[W_{t-1+k\Delta_l}^i - W_{t-1+(k-1)\Delta_l}^i\right]
\end{align*}
where
\begin{eqnarray*}
\overline{\xi}_m(X_{t-1+(k-1)\Delta_l}^i,\mu_{t-1+(k-1)\Delta_l}^{l,N}) & = & \frac{1}{N}\sum_{j=1}^N \xi_m(X_{t-1+(k-1)\Delta_l}^{l,i},X_{t-1+(k-1)\Delta_l}^{l,j})\quad m\in\{1,2\}\\
\mu_{t-1+(k-1)\Delta_l}^{l,N}(dx) & = & \frac{1}{N}\sum_{j=1}^N\delta_{\{X_{t-1+(k-1)\Delta_l}^{l,j}\}}(dx).
\end{eqnarray*}
Set $k=k+1$, if $k=\Delta_l^{-1}+1$ go to step 3.~otherwise go to the start of step 2..}
\item{For $i\in\{1,\dots,N\}$ compute:
\begin{align*}
\widetilde{X}_{t-1+k\Delta_{l-1}}^{l-1,i} & =  \widetilde{X}_{t-1+(k-1)\Delta_{l-1}}^{l-1,i} + a\left(\widetilde{X}_{t-1+(k-1)\Delta_{l-1}}^{l-1,i},\overline{\xi}_1(\widetilde{X}_{t-1+(k-1)\Delta_{l-1}}^{l-1,i},\widetilde{\mu}_{t-1+(k-1)\Delta_{l-1}}^{l-1,N})\right) + \\ & b\left(\widetilde{X}_{t-1+(k-1)\Delta_{l-1}}^{l-1,i},\overline{\xi}_2(\widetilde{X}_{t-1+(k-1)\Delta_{l-1}}^{l-1,i},\widetilde{\mu}_{t-1+(k-1)\Delta_{l-1}}^{l-1,N})\right)\left[W_{t-1+k\Delta_{l-1}}^i - W_{t-1+(k-1)\Delta_{l-1}}^i\right]
\end{align*}
where
\begin{eqnarray*}
\overline{\xi}_m(\widetilde{X}_{t-1+(k-1)\Delta_{l-1}}^{l-1,i},\widetilde{\mu}_{t-1+(k-1)\Delta_{l-1}}^{l-1,N}) & = & \frac{1}{N}\sum_{j=1}^N \xi_m(\widetilde{X}_{t-1+(k-1)\Delta_{l-1}}^{l-1,i},\widetilde{X}_{t-1+(k-1)\Delta_{l-1}}^{l-1,j})\quad m\in\{1,2\}\\
\widetilde{\mu}_{t-1+(k-1)\Delta_{l-1}}^{l-1,N}(dx) & = & \frac{1}{N}\sum_{j=1}^N\delta_{\{\widetilde{X}_{t-1+(k-1)\Delta_l}^{l-1,j}\}}(dx)
\end{eqnarray*}
and the increments of the Brownian motion $\left[W_{t-1+k\Delta_{l-1}}^i - W_{t-1+(k-1)\Delta_{l-1}}^i\right]$ were generated in step 2..
Set $k=k+1$, if $k=\Delta_{l-1}^{-1}+1$ go to step 4.~otherwise go to the start of step 3..}
\item{Output all the required laws $\mu_{t-1+\Delta_l}^{l,N},\dots,\mu_{t}^{l,N}$,  $\widetilde{\mu}_{t-1+\Delta_l}^{l-1,N},\dots,\widetilde{\mu}_{t}^{l-1,N}$.}
\end{enumerate}
\caption{Approximating the Consecutive Laws when starting with a particle approximation at time $t-1$, $t\in\{1,\dots,T\}$.}
\label{alg:basic_method_coup}
\end{algorithm}

\begin{algorithm}
\begin{enumerate}
\item{Input $M\in\mathbb{N}$ the cardinality of the state-space and two positive probability mass functions
$V_1^1,\dots,V_1^M$ and $V_2^1,\dots,V_2^M$ on $\{1,\dots,M\}$.  Go to 2..}
\item{Sample $U\sim\mathcal{U}_{[0,1]}$ (continuous uniform distribution on $[0,1]$). If $U<\sum_{i=1}^M\min\{V_1^i,V_2^i\}$ go to 3.~otherwise go to 4..}
\item{Sample an index $i_1$ using the probability mass function
$$
\mathbb{P}(i_1) = \frac{\min\{V_1^{i_1},V_2^{i_1}\}}{\sum_{j_1=1}^M\min\{V_1^{j_1},V_2^{j_1}\}}
$$
set $i_2=i_1$ and go to 5..}
\item{Sample the indices $(i_1,i_2)$ using the probability mass function
$$
\mathbb{P}(i_1,i_2) = \frac{V_1^{i_1}-\min\{V_1^{i_1},V_2^{i_1}\}}{1-\sum_{j_1=1}^M\min\{V_1^{j_1},V_2^{j_1}\}}
\frac{V_2^{i_2}-\min\{V_1^{i_2},V_2^{i_2}\}}{1-\sum_{j_1=1}^M\min\{V_1^{j_1},V_2^{j_1}\}}
$$
and go to 5..}
\item{Return the indices $(i_1,i_2)\in\{1,\dots,M\}^2$.}
\end{enumerate}
\caption{Simulating a Maximal Coupling.}
\label{alg:max_couple}
\end{algorithm}

\begin{algorithm}
\begin{enumerate}
\item{Input $l\in\mathbb{N}_0$ the level of discretization, $N\in\mathbb{N}$ the number of particles, $M$ the number of particles to compute laws of the process,  $T\in\mathbb{N}$ the final time and $X_0^{l,i}=\widetilde{X}_{0}^{l-1,i}=x_0$, $i\in\{1,\dots,N\}$. Set $t=1$.}% and $\mu_{t-1}^{N,M}(g) = 1$.}
\item{Run Algorithm \ref{alg:basic_method_coup} at time $t$ with $M$ particles and discretization level $l$ and, if needed and possible, using the input empirical measure that has been computed at the last call of Algorithm \ref{alg:basic_method_coup}.}
\item{For $(i,k)\in\{1,\dots,N\}\times\{1,\dots,\Delta_l^{-1}\}$ generate:
\begin{align}
X_{t-1+k\Delta_l}^{l,i} & =  X_{t-1+(k-1)\Delta_l}^{l,i} + a\left(X_{t-1+(k-1)\Delta_l}^{l,i},\overline{\xi}_1(X_{t-1+(k-1)\Delta_l}^{l,i},\mu_{t-1+(k-1)\Delta_l}^{l,M})\right) + \nonumber\\  &b\left(X_{t-1+(k-1)\Delta_l}^{l,i},\overline{\xi}_2(X_{t-1+(k-1)\Delta_l}^{l,i},\mu_{t-1+(k-1)\Delta_l}^{l,M})\right)\left[W_{t-1+k\Delta_l}^i - W_{t-1+(k-1)\Delta_l}^i\right]\label{eq:fine}
\end{align}
where the laws $\mu_{t-1+\Delta_l}^{l,M},\dots,\mu_{t-\Delta_l}^{l,M}$
 have been computed in step 2..  In addition for 
$(i,k)\in\{1,\dots,N\}\times\{1,\dots,\Delta_{l-1}^{-1}\}$ compute:
\begin{align*}
\widetilde{X}_{t-1+k\Delta_{l-1}}^{l-1,i} & =  \widetilde{X}_{t-1+(k-1)\Delta_{l-1}}^{l-1,i} + a\left(\widetilde{X}_{t-1+(k-1)\Delta_{l-1}}^{l-1,i},\overline{\xi}_1(\widetilde{X}_{t-1+(k-1)\Delta_{l-1}}^{l-1,i},\widetilde{\mu}_{t-1+(k-1)\Delta_{l-1}}^{l-1,M})\right) + \\ &b\left(\widetilde{X}_{t-1+(k-1)\Delta_{l-1}}^{l-1,i},\overline{\xi}_2(\widetilde{X}_{t-1+(k-1)\Delta_{l-1}}^{l-1,i},\widetilde{\mu}_{t-1+(k-1)\Delta_{l-1}}^{l-1,M})\right)\left[W_{t-1+k\Delta_{l-1}}^i - W_{t-1+(k-1)\Delta_{l-1}}^i\right]
\end{align*}
where the laws $\widetilde{\mu}_{t-1+\Delta_{l-1}}^{l-1,M},\dots,\widetilde{\mu}_{t-\Delta_{l-1}}^{l-1,M}$
 have been computed in step 2. and the increments of the Brownian motion $\left[W_{t-1+k\Delta_{l-1}}^i - W_{t-1+(k-1)\Delta_{l-1}}^i\right]$ were already generated in \eqref{eq:fine}.}
\item{For $i\in\{1,\dots,N\}$ compute the weights
$$
V^{l,i}_t = \frac{G(X_{t}^{l,i},y_t)}{\sum_{j=1}^N G(X_{t}^{l,i},y_t)} \quad
\widetilde{V}^{l-1,i}_t = \frac{G(\widetilde{X}_{t}^{l-1,i},y_t)}{\sum_{j=1}^N G(\widetilde{X}_{t}^{l-1,i},y_t)}
$$
and 
$$
\pi_{t}^{l,N,M}(\varphi) - \pi_{t}^{l-1,N,M}(\varphi) := \sum_{i=1}^N\left\{V^{l,i}_t\varphi(X_{t}^{l,i})-
\widetilde{V}^{l-1,i}_t\varphi(\widetilde{X}_{t}^{l-1,i})\right\}.
$$ 
For $i\in\{1,\dots,N\}$ sample indices
$(j_1^i,j_2^i)\in\{1,\dots,N\}^2$ using Algorithm \ref{alg:max_couple} with the probability mass functions $V^{l,1}_t,\dots,V^{l,N}_t$, $\widetilde{V}^{l-1,1}_t,\dots,\widetilde{V}^{l-1,N}_t$ 
and set 
$\check{X}_t^{l,i}=X_t^{l,j_1^i}$,  $\hat{X}_t^{l-1,i}=\widetilde{X}_t^{l-1,j_2^i}$.
For $i\in\{1,\dots,N\}$, set $X_t^{l,i}=\check{X}_t^{l,i}$,  $\widetilde{X}_t^{l-1,i}=\hat{X}_t^{l-1,i}$.
Set $t=t+1$ and if $t=T+1$ go to step 4.~otherwise go to step 2..}
\item{Return the estimators $\pi_{1}^{l,N,M}(\varphi) - \pi_{1}^{l-1,N,M}(\varphi),\dots,\pi_{T}^{l,N,M}(\varphi) - \pi_{T}^{l-1,N,M}(\varphi)$.}
\end{enumerate}
\caption{Coupled Particle Filter for approximating $\pi_{t}^l(\varphi)-\pi_{t}^{l-1}(\varphi)$.}
\label{alg:cpf}
\end{algorithm}

\section{Theoretical Analysis}\label{sec:theory}
\subsection{Notation}
Define the sets $\mathbb{N}=\{1,\dots\}$ and $\mathbb{N}_0=\{0,1,\dots\}$. Denote by $\mathcal{B}_b(\mathbb{R}^{d_1},\mathbb{R}^{d_2})$ the set of all $\mathbb{R}^{d_2}$ valued bounded Borel measurable functions defined on $\mathbb{R}^{d_1}$, for $k\in\mathbb{N}$ denote by $\mathcal{C}^k_b(\mathbb{R}^{d_1},\mathbb{R}^{d_2})$ the set of $k$-times continuously differentiable $\mathbb{R}^{d_2}$-valued functions whose domain is $\mathbb{R}^{d_1}$ and whose derivatives of order at most $k$ with respect to any components are bounded. Define $\mathcal{C}^{\infty}_b(\mathbb{R}^{d_1},\mathbb{R}^{d_2})=\cap_{k=1}^{\infty}\mathcal{C}^k_b(\mathbb{R}^{d_1},\mathbb{R}^{d_2})$. We use the shortcuts $\mathcal{C}_b^k(\mathbb{R}^d)=\mathcal{C}_b^k(\mathbb{R}^d,\mathbb{R})$ and $\mathcal{B}_b(\mathbb{R}^d)=\mathcal{B}_b(\mathbb{R}^d,\mathbb{R})$.
For $\varphi\in\mathcal{B}_b(\mathbb{R}^d)$ define the norm $|\varphi|_0=\underset{x\in\mathbb{R}^d}{\textup{ess}\sup}|\varphi(x)|$ and for $\varphi\in\mathcal{C}_b^k(\mathbb{R}^{d})$ define the norms $$|\varphi|_k=\sum_{i=0}^k\sup_{(j_1,\dots,j_i)\in\{1,\dots,d_1\}^i}\sup_{x\in\mathbb{R}^d}\left|\frac{\partial^i}{\partial x_{j_1}\dots\partial x_{j_i}}\varphi(x)\right|,$$
where $x=(x_1,\dots,x_{d_1})$. For a vector $v\in\mathbb{R}^d$ denote by $\|v\|$ the usual Euclidean norm.

\subsection{Assumptions}
\begin{hypA}\label{assump:A1}
The functions $a\in\mathcal{C}^2_b(\mathbb{R}^d\times\mathbb{R},\mathbb{R}^d)\cap \mathcal{B}_b(\mathbb{R}^d\times\mathbb{R},\mathbb{R}^d)$, $b\in\mathcal{C}^2_b(\mathbb{R}^d\times\mathbb{R},\mathbb{R}^{d\times d})\cap \mathcal{B}_b(\mathbb{R}^d\times\mathbb{R},\mathbb{R}^{d\times d})$, and $\xi_1,\xi_2\in\mathcal{C}^2_b(\mathbb{R}^d\times\mathbb{R}^{d})\cap \mathcal{B}_b(\mathbb{R}^d\times\mathbb{R}^{d})$, and $$\inf_{x\in\mathbb{R}^{d+1}}\inf_{v\in\mathbb{R}^{d}\backslash\{0\}} v^{\top}b(x)^{\top}b(x)v/\|v\|^2 > 0.$$
\end{hypA}
\begin{hypA}\label{assump:A2}
The function $G\in\mathcal{C}_b^2(\mathbb{R}^d)\cap\mathcal{B}_b(\mathbb{R}^d)$ and satisfies $\inf_{(x,y)\in\mathbb{R}^{d\times d_y}}G(x,y)>0$.
\end{hypA}
\begin{hypA}\label{assump:A3}
	The sequence of positive integers $N_0,\dots,N_L$ used to define the estimator $\widehat{\pi^L}$ is strictly decreasing and satisfies $N_{l} - N_{l+1}\geq L-l$ for $0\leq l\leq L-1$.
\end{hypA}
\begin{hypW}\label{assump:W1}
For every function $\varphi\in\mathcal{C}_b^{\infty}(\mathbb{R}^d)\cap\mathcal{B}_b(\mathbb{R}^d)$ the function $x\mapsto P_{\mu_{t-1},t}(\varphi)=\int \varphi(z)P_{\mu_{t-1},t}(x,dz)$ belongs to $\mathcal{C}_b^{\infty}(\mathbb{R}^d)$.
\end{hypW}
\begin{hypW}\label{assump:W2}
The functions $a\in\mathcal{C}^{\infty}_b(\mathbb{R}^d\times\mathbb{R},\mathbb{R}^d)$, $b\in\mathcal{C}^{\infty}_b(\mathbb{R}^d\times\mathbb{R},\mathbb{R}^{d\times d})$, and $\xi_1,\xi_2\in\mathcal{C}^{\infty}_b(\mathbb{R}^d\times\mathbb{R}^{d})$. Futhermore there exist functions $\zeta_1,\zeta_2\in\mathcal{C}^{\infty}_b(\mathbb{R}^d)$ such that $\xi_i(x,y)=\zeta_i(x-y)$ for all $x,y\in\mathbb{R}^d$ and $i\in\{1,2\}$.
\end{hypW}

\subsection{Main Theorems}\label{sec:results}
\begin{theorem}\label{thm:single_level_error}
Assume \hyperref[assump:A1]{(A1-2)} and \hyperlink{assump:W1}{(W1-2)}. Then for any $\varphi\in\mathcal{C}^{\infty}_b(\mathbb{R}^d)\cap\mathcal{B}_b(\mathbb{R}^d)$ and $t\in\{1,\dots,T\}$ there exists a constant $C<+\infty$ such that for every $(L,M,N)\in\mathbb{N}_0\times \mathbb{N}\times \mathbb{N}$:
	\begin{equation*}
		\mathbb{E}\left[ (\pi_t(\varphi) - \pi_t^{L,M,N}(\varphi))^2 \right] \leq C\left(\Delta_L^2+\frac{1}{N} + \frac{1}{M}\right).
	\end{equation*}
\end{theorem}

\begin{theorem}\label{thm:ml_error}
	Assume \hyperref[assump:A1]{(A1-3)} and \hyperlink{assump:W1}{(W1-2)}. Then for any $\varphi\in\mathcal{C}^{\infty}_b(\mathbb{R}^d)\cap\mathcal{B}_b(\mathbb{R}^d)$ and $t\in\{1,\dots,T\}$ there exists $C<+\infty$ such that for every $L\in\mathbb{N}_0$, and sequences of positive integers $M_0,\dots,M_L$ and $N_0,\dots,N_L$:
	\begin{equation*}
		\mathbb{E}\left[ (\pi_t(\varphi)-\widehat{\pi^L_t(\varphi)})^2 \right]\leq C\left(\Delta_L^2 + \sum_{l=0}^L\frac{\Delta_l^{1/2}}{N_l} + L\sum_{l=0}^L\frac{\Delta_l}{M_l}\right).
	\end{equation*}
\end{theorem}
Theorems \ref{thm:single_level_error} and \ref{thm:ml_error} are proven in the appendix. Assumptions \hyperref[assump:W1]{(W1-2)} are not needed to bound the variance of the estimates $\pi_t^{L,M,N}(\varphi)$ and $\widehat{\pi_t^L(\varphi)}$.

\subsection{Comments on Computational Cost}

Consider the estimator $\pi^{L,M,M}_t(\varphi)$, to achieve an error of order $\mathcal{O}(\epsilon^2)$ for $\epsilon>0$. Theorem \ref{thm:single_level_error} provides the guideline of selecting $L\propto |\log(\epsilon)|$ and $M,N\propto \epsilon^{-2}$, this choice leads to a cost of order $\mathcal{O}(M(M+N)\Delta_L) = \mathcal{O}(\epsilon^{-5})$. The multi-level estimator $\widehat{\pi_t^L(\varphi)}$ can achieve the same order of error $\mathcal{O}(\epsilon^2)$ at a lower cost. By utilizing the right hand side of the inequality in Theorem \ref{thm:ml_error} as a constraint for the error and optimizing the cost given this error constraint using the Lagrange multipliers method, we arrive to the following choices $L\propto |\log(\epsilon)|$, $M_l\propto \epsilon^{-2}|\log(\epsilon)|\Delta_l^{5/6}$ and $N_l\propto \epsilon^{-2}\Delta_l^{1/3}$, the resulting cost using these choices is $\mathcal{O}(\epsilon^{-4}|\log(\epsilon)|^2)$.  In certain scenarios, we believe it is possible to reduce the cost to $\mathcal{O}(\epsilon^{-4})$. 
This conjecture should be
justified by results that show better bounds on the weak error of the output law of Algorithm \ref{alg:basic_method} \cite{chass_2019}, this by itself is not enough, but one expects that there are scenarios
where the bound in Theorem \ref{thm:ml_error} can be improved to the reduced error $\mathcal{O}\left(\Delta_L^2 + \sum_{l=0}^L\frac{\Delta_l^{1/2}}{N_l} +\sum_{l=0}^L\frac{\Delta_l}{M_l}\right)$. This reduced error allows us to reduce the cost to $\mathcal{O}(\epsilon^{-4})$.
Assumption \hyperref[assump:A3]{(A3)} states that the number of particles should decrease relatively quickly. 
The increase rate stated in the assumption is very mild and not restrictive at all.  In fact $N_0,\dots,N_L$ are expected to decrease exponentially which is the case for the optimal choice discussed above.

\section{Numerical Results}\label{sec:numerics}
\subsection{Models}
\subsubsection*{Kuramoto Model}
We consider the Kuramoto model $\{X_t\}_{t\geq0}$ defined by the Mckean-Valsov SDE
\begin{equation*}
        dX_t = \left(\theta + \int\sin(X_t - y)d\mu_t(y)\right)dt + \sigma dW_t,\quad X_0 = x_0\in\mathbb{R},\  t\in[0,T],
\end{equation*}
where $\mu_t$ is the law of $X_t$, $\theta$ be a random variable, $\sigma>0$, and $T\in\mathbb{N}$. The Kuramoto model is a well-known mean field game model and it is frequently used for numerical demonstrations. 
Note that we expect that we can obtain the improved cost of $\mathcal{O}(\epsilon^{-4})$ in this case, due to the fact that the diffusion coefficient is constant.

\subsubsection*{Modified Kuramoto Model}
In this demonstration consider the process ${X_t}_{t\geq0}$ the Mckean-Vlasov SDE
\begin{equation*}
        dX_t = \left(\theta + \int\sin(X_t - y)d\mu_t(y)\right)dt + \left(\frac{\sigma}{1+X_t^2}\right)dW_t,\quad X_0 = x_0\in\mathbb{R},\  t\in[0,T],
\end{equation*}
where $\mu_t$ is the law of $X_t$, $\theta$ be a random variable, $\sigma>0$, and $T\in\mathbb{N}$. We call it the modified Kuramoto Model.  We consider this case as we expect that the non-constant diffusion coefficient
should demonstrate our theorerical results in the previous section, in that the computational cost should be
 $\mathcal{O}(\epsilon^{-4}|\log(\epsilon)|^2)$.

\subsection{Simulation Results}
For both models we impose observations $\{Y_k\}_{t=1}^T$ at unit times that follow $Y_k|X_k=x_k \sim \mathcal{N}(x_k, \tau^2)$ for $k\in\{1,\dots,T\}$ and set $x_0=1$, $\theta=0$, $\sigma=0.2$, $\tau=1$, and $T=50$. We simulated a path from each model using Algorithm \ref{alg:basic_method} with level $9$ and $5000$ particles. We compute the estimator $\pi^{L,M,N}_T(\varphi)$ and $\widehat{\pi^L_T(\varphi)}$ for the function $\varphi(x)=x$ where $M,N$ are chosen as discussed in subsection \ref{sec:results}. In order to calculate the mean square errors (MSE) of the two estimators we considered a proxy ground truth for the filter, this ground truth was obtained by computing $\pi^{L,M,N}_T(\varphi)$ at level $7$. The MSEs in our simulation were calculated using $128$ independent replications. 

Figure \ref{fig:kuramoto} shows log-log plots for the mean square errors against the theoretical costs for the two filter estimators applied to filtering the Kuramoto model and the modified Kuramoto.
In both cases we expect the multilevel methods to have rates of about -2 (Kuramot0) and -2.2,
reflecting the costs of $\mathcal{O}(\epsilon^{-4})$ and  $\mathcal{O}(\epsilon^{-4}|\log(\epsilon)|^2)$,
 and this is exactly what we observe.  For the case of the PF (single level) the rate should be about -2.5, which
gives the cost  $\mathcal{O}(\epsilon^{-5})$.  The plots show the expected improvements of the multilevel method.

% \begin{figure}
%     \centering
%     \includegraphics[width=0.5\linewidth]{Kur_f.pdf}
%     \caption{Kuramoto Model: SL: Single level for $2\leq L\leq 6$, ML: Multi-level for $3\leq L\leq 6$.}
%     \label{fig:kuramoto}
% \end{figure}
% \begin{figure}
%     \centering
%     \includegraphics[width=0.5\linewidth]{ModKur_f.pdf}
%     \caption{Modified Kuramoto Model: SL: Single level for $2\leq L\leq 6$, ML: Multi-level for $3\leq L\leq 6$.}
%     \label{fig:mod_kuramoto}
% \end{figure}

\begin{figure}
    \includegraphics[width=0.5\linewidth]{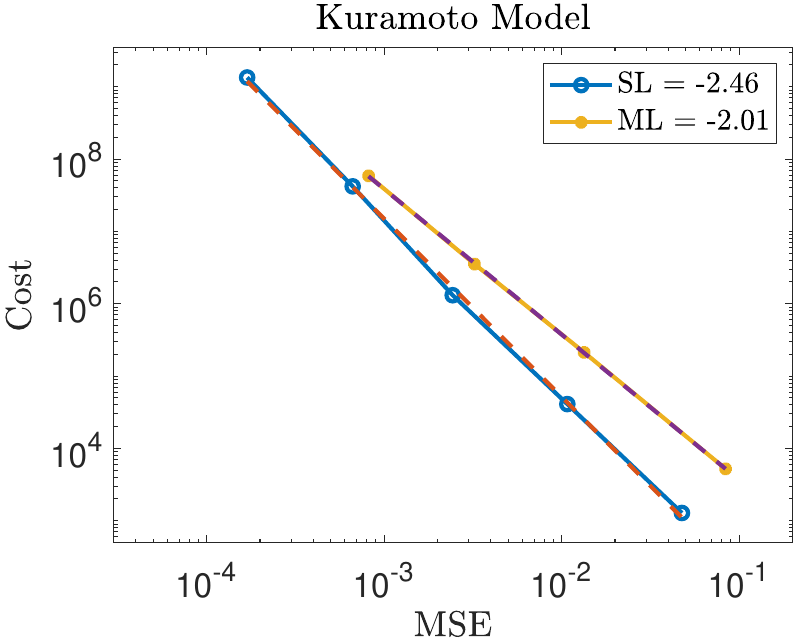}
    \includegraphics[width=0.5\linewidth]{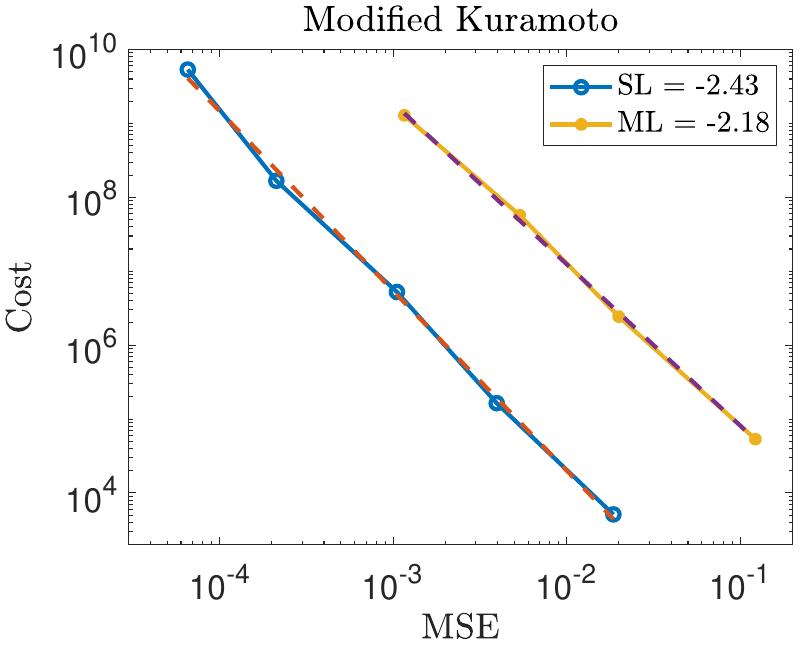}
    \caption{Left: Kuramoto Model: Single level (SL) for $2\leq L\leq 6$, Multi-level (ML) for $3\leq L\leq 6$.\\ Right: Modified Kuramoto Model: Single level (SL) for $2\leq L\leq 6$, Multi-level (ML) for $3\leq L\leq 6$.}
    \label{fig:kuramoto}
\end{figure}

\subsubsection*{Acknowledgements}

AJ was supported by start-up funding at CUHK-SZ.

\appendix
\section{Proofs}
\subsection{Structure}
%The following appendix contains several technical results that are used to prove Theorems \ref{thm:single_level_error} and \ref{thm:ml_error}.  We begin by a short section on the additional notations that are to be used in the appendix.  Section \ref{appen:proof_sl} gives several technical results which conclude with the proof of Theorem \ref{thm:single_level_error}. Section \ref{appen:proof_ml} again gives a collection of technical results which culminate with the proof of Theorem \ref{thm:ml_error}. To understand the appendix, it should be read in chronological order.
The following appendix contains several technical results that are used to prove Theorems \ref{thm:single_level_error} and \ref{thm:ml_error}.  We begin by a short section on the additional notations and auxiliary processes that are to be used in the appendix. Section \ref{appen:Z_l_M_increments_bounds} is devoted to studying the auxiliary processes $\{Z^{l,M,i}\}_{i=1}^M$, which are defined below. Sections \ref{appen:Z_l_M_increments_bounds} and \ref{appen:pi_increments_bounds} provide technical results leading to the proof of Lemma \ref{lm:pi_ml}. Section \ref{appen:proof_sl} introduces further several technical results and employs Lemma \ref{lm:pi_ml} to ultimately prove Theorem \ref{thm:single_level_error}. Section \ref{appen:proof_ml} presents a technical discussion and uses the previously proven results in order to present a proof for Theorem \ref{thm:ml_error}. To understand the appendix, it should be read in chronological order.

\subsection{Notation}
Let $(\Omega,\mathcal{F},\mathbb{P})$ be a measure space, $W_t$ a Brownian motion on $\Omega$, and $X_t$ be the solution of SDE \eqref{eq:sde}. For random variables $A_1,\dots,A_n$ and $\sigma$-subalgebras $\mathcal{F}_1,\dots\mathcal{F}_n$ we define $\sigma(A_1,\dots,A_n,\mathcal{F}_1,\dots,\mathcal{F}_n)$ to be the smallest $\sigma$-subalgebra $\mathcal{G}$ of $\mathcal{F}$ that contains the set $\bigcup_{i=1}^n\mathcal{F}_i$ and that satisfies that the random variables $A_1,\dots,A_n$ are measurable with respect to $\mathcal{G}$. We consider step 2. of Algorithm \ref{alg:cpf}: This step calls Algorithm \ref{alg:basic_method_coup}, and step 2. of Algorithm \ref{alg:cpf} is called $T$ times. For $(l,M,N)\in\mathbb{N}^3$ given, we then use the associated notation:
\sloppy
\begin{itemize}
%\begin{compactitem}
    \item $\{\mu_{k\Delta_l}^{l,M}\}_{k=0}^{k=T\Delta_l^{-1}}$ and $\{\mu_{k\Delta_{l-1}}^{l-1,M}\}_{k=0}^{k=T\Delta_{l-1}^{-1}}$ are the output measures of all the iterations of step 2. in \ref{alg:cpf}.
    \item $\{\{Z_{k\Delta_l}^{l,M,i}\}_{k=0}^{k=T\Delta_l^{-1}}\}_{i=1}^M$ and $\{\{Z_{k\Delta_{l-1}}^{l-1,M,i}\}_{k=0}^{k=T\Delta_{l-1}^{-1}}\}_{i=1}^M$ are the particles that generated the outputs of all the iterations of step 2. in \ref{alg:cpf}. Let $\mathcal{F}^{l,M}_{k\Delta_l} = \sigma(\{\{Z_{j\Delta_l}^{l,M,i}\}_{j=0}^{j=k}\}_{i=1}^M)$ and $\mathcal{F}^{l-1,M}_{k\Delta_{l-1}} =\sigma(\{\{Z_{j\Delta_{l-1}}^{l-1,M,i}\}_{j=0}^{j=k}\}_{i=1}^M)$ be the natural filterations of these particles respectively.
    \item $\{B^i_t\}_{i=1}^M$ are the Brownian motions generated in all the iterations of step 2. in Algorithm \ref{alg:cpf} in order to generate the particles $\{\{Z_{k\Delta_l}^{l,M,i}\}_{k=1}^{k=T\Delta_l^{-1}}\}_{i=1}^M$ and $\{\{Z_{k\Delta_{l-1}}^{l-1,M,i}\}_{k=1}^{k=T\Delta_{l-1}^{-1}}\}_{i=1}^M$. Denote $\Delta B^{l,i}_t = B^{i}_t - B^{i}_{t-\Delta_l}$ for $t\in[\Delta_l,T]$ and $i\in\{1,\dots,M\}$. Similarly denote $\Delta W^l_t = W_t - W_{t-\Delta_l}$ for $t\in[\Delta_l,T]$.
    \item For $s\in\{l-1,l\}$ define the discrete process $\{\{Z_{k\Delta_s}^{s,i}\}_{k=0}^{k=T\Delta_l^{-1}}\}_{i=1}^M$ as the solution of the equations
        % \begin{eqnarray}
        % Z_{k\Delta_s}^{s,i} & = & Z_{(k-1)\Delta_s}^{s,i} + a\left(Z_{(k-1)\Delta_s}^{s,i},\overline{\xi}_1(Z_{(k-1)\Delta_s}^{s,i},\mu_{(k-1)\Delta_l}^{s})\right) +\nonumber \\ & &b\left(Z_{(k-1)\Delta_s}^{s,i},\overline{\xi}_2(Z_{(k-1)\Delta_s}^{s,i},\mu_{(k-1)\Delta_l}^{s})\right)\Delta B^{s,i}_{k\Delta_s}\label{eq:Z_l_i},
        % \end{eqnarray}
        \begin{eqnarray}
        Z_{(k+1)\Delta_s}^{s,i}  =  Z_{k\Delta_s}^{s,i} + a\left(Z_{k\Delta_s}^{s,i},\overline{\xi}_1(Z_{k\Delta_s}^{s,i},\mu_{k\Delta_l}^{s})\right) + b\left(Z_{k\Delta_s}^{s,i},\overline{\xi}_2(Z_{k\Delta_s}^{s,i},\mu_{k\Delta_l}^{s})\right)\Delta B^{s,i}_{(k+1)\Delta_s}\label{eq:Z_l_i},
        \end{eqnarray}
        with $\mu^s_{k\Delta_s}$ being the law of $Z^{s,1}_{k\Delta_s}$ and with the initial condition $Z^{s,i}_0=x_0$ for every $(k,i)\in\{0,\dots,T\Delta_s^{-1}-1\}\times\{1,\dots,M\}$.
    %\item $W_t$ is a Brownian motion independent of the Brownian motions $\{B^i_t\}_{i=1}^M$. Denote $\Delta W^{l}_t = W_t - W_{t-\Delta_l}$ for $t\in[\Delta_l,T]$.
    \item For $s\in\{l-1,l\}$ define the processes $\tilde{Z}^{s,M}_t$, and $\tilde{Z}^{s}_t$, as the solutions of 
            \begin{eqnarray}\label{eq:Z_tilde_1}
            \tilde{Z}_{(k+1)\Delta_s}^s  =  \tilde{Z}_{k\Delta_s}^s + a\left(\tilde{Z}_{k\Delta_s}^s,\overline{\xi}_1(\tilde{Z}_{k\Delta_s}^s,\mu_{k\Delta_s}^s)\right) +b\left(\tilde{Z}_{k\Delta_s}^s,\overline{\xi}_2(\tilde{Z}_{k\Delta_s}^s,\mu_{k\Delta_s}^s)\right)\Delta W^{s}_{(k+1)\Delta_s},\\
            \tilde{Z}_{(k+1)\Delta_s}^{s,M} =  \tilde{Z}_{k\Delta_s}^{s,M} + a\left(\tilde{Z}_{k\Delta_s}^{s,M},\overline{\xi}_1(\tilde{Z}_{k\Delta_s}^{s,M},\mu_{k\Delta_s}^{s,M})\right) +b\left(\tilde{Z}_{k\Delta_s}^{s,M},\overline{\xi}_2(\tilde{Z}_{k\Delta_s}^{s,M},\mu_{k\Delta_s}^{s,M})\right)\Delta W^{s}_{(k+1)\Delta_s},
            \end{eqnarray}
            with $\mu_{k\Delta_s}^s$ being the law of $Z_{k\Delta_s}^{s,1}$, $\mu_{k\Delta_s}^{s,M}=\frac{1}{M}\sum_{i=1}^M\delta_{Z_{k\Delta_s}^{s,M,i}}$ and with the initial condition $\tilde{Z}^{s}_0=\tilde{Z}^{s,M}_0=x_0$ for every $k\in\{0,\dots,T\Delta_{s}^{-1}-1\}$. Conditioning on the  $\sigma$-algebra $\mathcal{F}^{s,M}_t$ denote by $P_{\mu_{t-1}^{s,M},t}^{s,M}$ the transition kernel the process $\{\tilde{Z}^{s,M}_t\}_{t\in\{0,1,\dots,T\}}$.
    %\item For $s\in\{l-1,l\}$, $P_{\mu_{p-1}^{s,M},p}^{l,M}(x_{p-1},dx_p)$ is the conditional law of $\tilde{Z}^{s,M}_t$ given the sigma algebra $\sigma\left(\tilde{Z}^{s,M}_{t-1}, \mathcal{F}^{l,M}_t\mathcal{F}^{l-1,M}_t\right)$.
    \item For $\varphi\in\mathcal{B}(\mathbb{R}^d)$ define
        $$
        \gamma_t(\varphi) := \int_{\mathbb{R}^{dt}}\varphi(x_t)\left\{\prod_{p=1}^t G(x_p,y_p)\right\}\prod_{p=1}^tP_{\mu_{p-1},p}(x_{p-1},dx_p),
        $$
        $$
        \gamma_t^l(\varphi) := \int_{\mathbb{R}^{dt}}\varphi(x_t)\left\{\prod_{p=1}^t G(x_p,y_p)\right\}\prod_{p=1}^tP_{\mu_{p-1}^l,p}^l(x_{p-1},dx_p),
        $$
        and for $s\in\{l-1,l\}$ define
        $$
        \gamma_t^{l,M}(\varphi) := \int_{\mathbb{R}^{dt}}\varphi(x_t)\left\{\prod_{p=1}   ^t G(x_p,y_p)\right\}\prod_{p=1}^tP_{\mu_{p-1}^{l,M},p}^{l,M}(x_{p-1},dx_p).
        $$
    \item Define $\pi^{l,M}_t(\varphi) = \gamma^{l,M}_t(\varphi)/\gamma^{l,M}_t(1)$.
%\end{compactitem}
\end{itemize}
The Brownian motions $\{B_t^i\}_{i=1}^M$ and $W_t$ are independent. Notice that $\gamma_t^{l,M}(\varphi)$ and $\gamma_t^{l-1,M}(\varphi)$ are random variables measurable with respect to $\sigma(\mathcal{F}^{l-1,M}_t,\mathcal{F}^{l,M}_t)$. The process $\tilde{Z}_t^s$ is the Euler-Maruyama discretization for the process $X_t$. We introduce the process $\tilde{Z}_t^{s,M}$ to approximate $\tilde{Z}_t^s$. The process $\tilde{Z}_t^{s,M}$ utilizes the particles $\{Z_t^{s,M,i}\}_{i=1}^M$ to create an empirical measure that approximates the measure $\mu^s_t$ independently of $W_t$. Our main objective is study the increments $\pi_t^{l,M,N}(\varphi)-\pi_t^{l-1,M,N}(\varphi)$. To achieve this, we investigate the increments $Z_t^{l,M,i}-Z_t^{l-1,M,i}$ and $\tilde{Z}_t^{l,M}-\tilde{Z}_t^{l-1,M}$. Finally we leverage the obtained results and the theory of multi-level particle filters to bound the $\mathbb{L}_2$ error of the estimator $\widehat{\pi_t^L(\varphi)}$.\\
Throughout the proofs below we will denote constants by $C$. For concise notation $C$ may change values from step to another. Unless explicitly stated $C$ is independent of $l$,$M$,$N$, and $i$.
%\subsection{Proof of Theorem \ref{thm:single_level_error}}
\subsection{Bounds on $Z^{l,M,i}_t$ increments}
\label{appen:Z_l_M_increments_bounds}
\begin{lem}\label{lm:technical_lm_1}
    Let $\Delta>0$ and let $\{u_{k\Delta}\}$ be a sequence of non-negative real numbers indexed by the set $\Delta\mathbb{N}_0=\{k\Delta:k\in\mathbb{N}_0\}$. Suppose that there exist constants $C_1,C_2>0$ for the which for every $k\in\mathbb{N}$ the following inequality holds:
    \begin{equation*}
        u_{k\Delta} \leq (1+C_1\Delta)u_{(k-1)\Delta} + C_2\Delta.
    \end{equation*}
    Then, the following inequality holds for every $k\in\mathbb{N}_0$
    \begin{equation*}
        u_{k\Delta} \leq e^{k\Delta C_1}(u_0 + k\Delta C_2).
    \end{equation*}
\end{lem}
\begin{lem}\label{lm:technical_lm_3}
    Let $k\in\mathbb{N}$, there exists $C>0$ such that for all $(A,B)\in(\mathbb{R}^d)^2$ The following inequality holds
    \begin{eqnarray}
        \|A+B\|^{2k} \leq \|A\|^{2k} + 2k\|A\|^{2k-2}A^TB + C\|B\|^2(\|A\|^{2k-2}+\|B\|^{2k-2}).
    \end{eqnarray}
\end{lem}

\begin{lem}\label{lm:discrete_mv_rate_2_part_1}
Assume \hyperref[assump:A1]{(A1)}. Let $k\in\mathbb{N}$. There exists a constant $C<+\infty$ such that for every $(l,M)\in\mathbb{N}_0\times\mathbb{N}$, $t\in\{0,\Delta_l,\dots,T\}$, and $i\in\{1,\dots,M\}$:
    \begin{equation*}
        \max\{\mathbb{E}[\|Z^{l,M,i}_t\|^k], \mathbb{E}[\|Z^{l,i}_t\|^k]\}\leq C,
    \end{equation*}
\end{lem}
\begin{proof}
    It is enough to prove the inequality for even positive integers $k$, the general case follows from the inequality $\mathbb{E}[\|A\|^{k}]\leq \mathbb{E}[\|A\|^{2k}]^{1/2}$ for any $\mathbb{R}^d$ valued random variable $A$. Using the definition \eqref{eq:Z_l_i} of $Z^{l,i}_t$, applying Lemma \ref{lm:technical_lm_3} twice, Assumption \hyperref[assump:A1]{(A1)} that the functions $a,b$ are bounded, $\Delta_l\leq 1$, $\mathbb{E}[\Delta B^{l,i}_{t-\Delta_l}]=0$, and $\mathbb{E}[\|B^{l,i}_{t-\Delta_l}\|^k]\leq C\Delta_l^{k/2}$ we have
    \begin{equation}\label{eq:expand_1}
        \begin{split}
            \mathbb{E}[\|Z^{l,i}_{t}\|^k] =\;& \mathbb{E}\left[ \|Z^{l,i}_{t-\Delta_l} + a(Z^{l,i}_{t-\Delta_l},\bar{\xi}_2(Z^{l,i}_{t-\Delta_l},\mu^{l}_{t-\Delta_l}))\Delta + b(Z^{l,i}_{t-\Delta_l},\bar{\xi}_2(Z^{l,i}_{t-\Delta_l},\mu^{l}_{t-\Delta_l}))\Delta B^{l,i}_{t-\Delta_l}\|^k\right]\\
            \leq\;& \mathbb{E}[\|Z^{l,i}_{t-\Delta_l}\|^k] + C\Delta_l\max_{r\in\{1,2,3,4\}}\mathbb{E}[\|Z^{l,i}_{t-\Delta_l}\|^{k-r}] + C\Delta_l^{k/2}\\
            \leq\;& (1+C\Delta_l)\mathbb{E}[\|Z^{l,i}_{t-\Delta_l}\|^{k}] + C\Delta_l.
        \end{split}
    \end{equation}
    \sloppy The last line follows from the inequality $x^s\leq 1+x^k$ for $0\leq s\leq k$ and $x\geq0$.
    Applying Lemma \ref{lm:technical_lm_1} to the sequence $\{[\|Z^{l,i}_{s\Delta_l}\|^{k}]\}_{s=0}^{T/\Delta_l}$ yields the inequality $\mathbb{E}[\|Z^{l,i}_{t-\Delta_l}\|^{k}]\leq C$, clearly following similar steps yields $\mathbb{E}[\|Z^{l,M,i}_{t-\Delta_l}\|^{k}]\leq C$. 
\end{proof}

\begin{lem}\label{lm:discrete_mv_rate_2_part_2}
    Assume \hyperref[assump:A1]{(A1)}. Let $k\in\mathbb{N}$. There exists a constant $C<+\infty$ such that for every $(l,M)\in\mathbb{N}\times\mathbb{N}$, $t\in\{0,\Delta_l,\dots,T\}$, and $i\in\{1,\dots,M\}$:
    \begin{equation*}
        \mathbb{E}\left[ \left\|Z^{l,M,i}_t-Z^{l,i}_t\right\|^k \right]\leq C\frac{1}{M^{k/2}},
    \end{equation*}
    \begin{equation*}
        \mathbb{E}\left[ \left\|Z^{l,M,i}_t-Z^{l-1,M,i}_t\right\|^k \right]\leq C\Delta_l^{k/2},
    \end{equation*}
\end{lem}
\begin{proof}
    Let $t\in \{\Delta,2\Delta,\dots,T\}$, we have.
	\begin{equation}\label{eq:Z_M_difference}
		\begin{split}
			Z^{l,M,i}_t - Z^{l,i}_t =\;& Z^{l,M,i}_{t-\Delta_l} - Z^{l,i}_{t-\Delta_l} \\+\;& (a(Z^{l,M,i}_{t-\Delta_l},\bar{\xi}_1(Z^{l,M,i}_{t-\Delta_l},\mu^{l,M}_{t-\Delta_l}))-a(Z^{l,i}_{t-\Delta_l},\bar{\xi}_1(Z^{l,i}_{t-\Delta_l},\mu^{l}_{t-\Delta_l})))\Delta_L \\ +\; & (b(Z^{l,M,i}_{t-\Delta_l},\bar{\xi}_2(Z^{l,M,i}_{t-\Delta_l},\mu^{l,M}_{t-\Delta_l}))-b(Z^{l,i}_{t-\Delta_l},\bar{\xi}_2(Z^{l,i}_{t-\Delta_l},\mu^{l}_{t-\Delta_l})))\Delta B^{l,i}_t.
		\end{split}
	\end{equation}
    Using the $C^p$ inequality and that the functions $a$ and $\xi_1$ are bounded and Lipchitz (Assumption \hyperref[assump:A1]{(A1)}) there exists a constant $C>0$ that satisfies
\begin{equation}\label{eq:a_Z_2}
	\begin{split}
		\;&\mathbb{E}\left[\|a(Z^{l,M,i}_{t-\Delta_l},\bar{\xi}_1(Z^{l,M,i}_{t-\Delta_l},\mu^{l,M,i}_{t-\Delta_l}))-a(Z^{l,i}_{t-\Delta_l},\bar{\xi}_1(Z^{l,i}_{t-\Delta_l},\mu^{l,i}_{t-\Delta_l}))\|^k\right]\\
		\leq \;& C\Bigg(\mathbb{E}\left[\|Z^{l,M,i}_{t-\Delta_l} - Z^{l,i}_{t-\Delta_l}\|^k\right]  + \mathbb{E}\Bigg[\Bigg|\frac{1}{M}\sum_{j=1}^M\xi_1(Z^{l,M,i}_{t-\Delta_l}Z^{l,M,j}_{t-\Delta_l})-\frac{1}{M}\sum_{j=1}^M\xi_1(Z^{l,i}_{t-\Delta_l},Z^{l,j}_{t-\Delta_l})\Bigg|^k\Bigg]\\
        +\;& \frac{1}{M^k}\mathbb{E}\Bigg[\Bigg|\sum_{\substack{j=1\\j\not=i}}^M\left(\xi_1(Z^{l,i}_{t-\Delta_l},Z^{l,j}_{t-\Delta_l})-\int \xi_1(Z^{l,i}_{t-\Delta_l},y)d\mu^l_{t-\Delta_l}(y)\right)\Bigg|^k\Bigg]  \Bigg) + C\frac{1}{M^k}.\\
  \end{split}
\end{equation}
    The second term after the inequality is bounded by $\frac{1}{M}\sum_{j=1}^M\mathbb{E}[|\xi_1(Z^{l,M,i}_{t-\Delta_l}Z^{l,M,j}_{t-\Delta_l})-\xi_1(Z^{l,i}_{t-\Delta_l},Z^{l,j}_{t-\Delta_l})|^k]\leq C\sup_i\mathbb{E}[\|Z^{l,M,i}_{t-\Delta_l} - Z^{l,i}_{t-\Delta_l}\|^k]$ by the $C^p$ inequality. %For the third term we condition on $Z^{l,i}_{t-\Delta_t}$ and use the Marcinkiewicz-Zygmund inequality (Lemma \ref{lm:discrete_mv_rate_2_part_1} guarantees that $Z^{l,i}_{t-\Delta_t}$ is $\mathbb{L}_k$):
    For the third term we use the idependence of $Z_t^{l,i}$ and $Z_t^{l,j}$ for $i\not=j$, Fubini Theorem, and Marcinkiewicz-Zygmund inequality: %(Lemma \ref{lm:discrete_mv_rate_2_part_1} guarantees that $Z^{l,i}_{t-\Delta_t}$ is $\mathbb{L}_k$):
    % \begin{equation}\label{eq:MZ_1}
    %     \begin{split}
    %          \;&\mathbb{E}\left[\left|\sum_{\substack{j=1\\j\not=i}}^M\left(\xi_1(Z^{l,i}_{t-\Delta_l},Z^{l,j}_{t-\Delta_l})-\int \xi_1(Z^{l,i}_{t-\Delta_l},y)d\mu^l_{t-\Delta_l}(y)\right)\right|^k \right]\\
    %          =\;&\mathbb{E}\left[\mathbb{E}\left[\left|\sum_{\substack{j=1\\j\not=i}}^M\left(\xi_1(Z^{l,i}_{t-\Delta_l},Z^{l,j}_{t-\Delta_l})-\int \xi_1(Z^{l,i}_{t-\Delta_l},y)d\mu^l_{t-\Delta_l}(y)\right)\right|^k\Bigg| Z^{l,i}_{t-\Delta_t}\right] \right]\\
    %          \leq\;& C\mathbb{E}\left[\mathbb{E}\left[ \left(\sum_{\substack{j=1\\j\not=i}}^M\left(\xi_1(Z^{l,i}_{t-\Delta_l},Z^{l,j}_{t-\Delta_l})-\int \xi_1(Z^{l,i}_{t-\Delta_l},y)d\mu^l_{t-\Delta_l}(y)\right)^2\right)^{k/2}\Bigg| Z^{l,i}_{t-\Delta_t}\right]\right]\\
    %          \leq\;& CM^{k/2},
    %     \end{split}
    % \end{equation}
    % \begin{equation}\label{eq:MZ_1}
    %     \begin{split}
    %          \;&\mathbb{E}\left[\left|\sum_{\substack{j=1\\j\not=i}}^M\left(\xi_1(Z^{l,i}_{t-\Delta_l},Z^{l,j}_{t-\Delta_l})-\int \xi_1(Z^{l,i}_{t-\Delta_l},y)d\mu^l_{t-\Delta_l}(y)\right)\right|^k \right]\\
    %          \leq\;& C\mathbb{E}\left[ \left(\sum_{\substack{j=1\\j\not=i}}^M\left(\xi_1(Z^{l,i}_{t-\Delta_l},Z^{l,j}_{t-\Delta_l})-\int \xi_1(Z^{l,i}_{t-\Delta_l},y)d\mu^l_{t-\Delta_l}(y)\right)^2\right)^{k/2}\right]\\
    %          \leq\;& CM^{k/2},
    %     \end{split}
    % \end{equation}
    \begin{equation}\label{eq:MZ_1}
        \begin{split}
             \;&\mathbb{E}\Bigg[\Bigg|\sum_{\substack{j=1\\j\not=i}}^M\left(\xi_1(Z^{l,i}_{t-\Delta_l},Z^{l,j}_{t-\Delta_l})-\int \xi_1(Z^{l,i}_{t-\Delta_l},y)d\mu^l_{t-\Delta_l}(y)\right)\Bigg|^k \Bigg]\\
             =\;&\int\mathbb{E}\Bigg[\Bigg|\sum_{\substack{j=1\\j\not=i}}^M\left(\xi_1(z,Z^{l,j}_{t-\Delta_l})-\int \xi_1(z,y)d\mu^l_{t-\Delta_l}(y)d\right)\Bigg|^k \Bigg]\mu^l_{t-\Delta_l}(z)\\
             \leq\;& C\int\mathbb{E}\Bigg[ \Bigg(\sum_{\substack{j=1\\j\not=i}}^M\left(\xi_1(z,Z^{l,j}_{t-\Delta_l})-\int \xi_1(z,y)d\mu^l_{t-\Delta_l}(y)\right)^2\Bigg)^{k/2}\Bigg]\mu^l_{t-\Delta_l}(z)\\
             \leq\;& CM^{k/2},
        \end{split}
    \end{equation}
    where the last line follows from the boundedness of $\xi_1$. Thus we have shown that
    \begin{equation}\label{eq:a_bound_1}
        \begin{split}
            \;&\mathbb{E}\left[\|a(Z^{l,M,i}_{t-\Delta_l},\bar{\xi}_1(Z^{l,M,i}_{t-\Delta_l},\mu^{l,M,i}_{t-\Delta_l}))-a(Z^{l,i}_{t-\Delta_l},\bar{\xi}_1(Z^{l,i}_{t-\Delta_l},\mu^{l,i}_{t-\Delta_l}))\|^k\right]
            \leq C\sup_i\mathbb{E}[\|Z^{l,M,i}_{t-\Delta_l} - Z^{l,i}_{t-\Delta_l}\|^k] + C\frac{1}{M^{k/2}}.
        \end{split}
    \end{equation}
    A similar bound holds for the function $b$. Analogously to the proof of Lemma \ref{lm:discrete_mv_rate_2_part_1}, we apply Lemma \ref{lm:technical_lm_3} twice on $\|Z^{l,M,i}_{t}-Z^{l,i}_{t}\|$ following equation \eqref{eq:Z_M_difference} and using the bound \eqref{eq:a_bound_1} (for both $a$ and $b$) yields
    \begin{equation*}
        \sup_i\mathbb{E}[\|Z^{l,M,i}_{t} - Z^{l,i}_{t}\|^k]\leq (1+C\Delta_l)\sup_i\mathbb{E}[\|Z^{l,M,i}_{t-\Delta_l} - Z^{l,i}_{t-\Delta_l}\|^k] + C\Delta_l\frac{1}{M^{k/2}}.
    \end{equation*}
    Applying Lemma \ref{lm:technical_lm_1} to the sequence $\{\sup_i\mathbb{E}[\|Z^{l,M,i}_{s\Delta_l} - Z^{l,i}_{s\Delta_l}\|^k]\}_{s=0}^{T/\Delta_l}$ and using $Z^{l,i}_0=Z^{l,M,i}_0=x_0$ proves the first inequality. For the second inequality we write
    \begin{equation}\label{eq:L_L+1_decomposition}
        \begin{split}
             &Z^{l,M,i}_{t}  - Z^{l-1,M,i}_{t} \\
             =\;& Z^{l,M,i}_{t-\Delta_{l-1}}  - Z^{l-1,M,i}_{t-\Delta_{l-1}} \\
             +\;& \left(a\left(Z^{l,M,i}_{t-\Delta_{l-1}},\bar{\xi}_{1}(Z^{l,M,i}_{t-\Delta_{l-1}}, \mu^{l,M}_{t-\Delta_{l-1}})\right) - a\left(Z^{l-1,M,i}_{t-\Delta_{l-1}},\bar{\xi}_{1}(Z^{l-1,M,i}_{t-\Delta_{l-1}}, \mu^{l-1,M}_{t-\Delta_{l-1}})\right)\right)\Delta_{l-1} \\
             +\;& \left(b\left(Z^{l,M,i}_{t-\Delta_{l-1}},\bar{\xi}_{2}(Z^{l,M,i}_{t-\Delta_{l-1}}, \mu^{l,M}_{t-\Delta_{l-1}})\right) - b\left(Z^{l-1,M,i}_{t-\Delta_{l-1}},\bar{\xi}_{2}(Z^{l-1,M,i}_{t-\Delta_{l-1}}, \mu^{l-1,M}_{t-\Delta_{l-1}})\right)\right)\Delta B^{l-1,i}_{t}\\
             +\;& \left(a\left(Z^{l,M,i}_{t-\Delta_{l}},\bar{\xi}_{1}(Z^{l,M,i}_{t-\Delta_{l}}, \mu^{l,M}_{t-\Delta_{l}})\right) - a\left(Z^{l,M,i}_{t-\Delta_{l-1}},\bar{\xi}_{1}(Z^{l,M,i}_{t-\Delta_{l-1}}, \mu^{l,M}_{t-\Delta_{l-1}})\right)\right)\Delta_{l} \\
             +\;& \left(b\left(Z^{l,M,i}_{t-\Delta_{l}},\bar{\xi}_{2}(Z^{l,M,i}_{t-\Delta_{l}}, \mu^{l,M}_{t-\Delta_{l}})\right) - b\left(Z^{l,M,i}_{t-\Delta_{l-1}},\bar{\xi}_{2}(Z^{l,M,i}_{t-\Delta_{l-1}}, \mu^{l,M}_{t-\Delta_{l-1}})\right)\right)\Delta B^{l,i}_{t}.
        \end{split}
    \end{equation}
    Bounding each term following the same technique used to prove the first inequality we can show the following the inequality
        \begin{equation*}
        \begin{split}
            &\sup_i\mathbb{E}\left[\|Z^{l,M,i}_{t}  - Z^{l-1,M,i}_{t}\|^k\right]\\
            \leq\;& (1+C\Delta_L)\sup_i\mathbb{E}\left[\|Z^{l,M,i}_{t-\Delta_{l-1}}  - Z^{l-1,M,i}_{t-\Delta_{l-1}}\|^k\right]
            + C\Delta_{l-1}\left(\sup_i\mathbb{E}\left[\|Z^{l,M,i}_{t-\Delta_{l}}  - Z^{l,M,i}_{t-\Delta_{l-1}}\|^k\right] + \Delta_{l-1}^{k/2}\right).\\
        \end{split}
    \end{equation*}
    From the definition of $Z^{l,M,i}_t$ and the boundedness of $a$ and $b$ we have $\sup_i\mathbb{E}[\|Z^{l,M,i}_{t-\Delta_{l}}  - Z^{l,M,i}_{t-\Delta_{l-1}}\|^k]\leq C\Delta_l^{k/2}$. The proof then follows similarly to the proof of the first inequality by applying Lemma \ref{lm:technical_lm_1} to the sequence $\{\sup_i\mathbb{E}[\|Z^{l,M,i}_{s\Delta_{l-1}}  - Z^{l-1,M,i}_{s\Delta_{l-1}}\|^k]\}_{s=0}^{T/\Delta_{l-1}}$.
\end{proof}

\begin{lem}\label{lm:technical_lm_2}
    Let $f\in\mathcal{C}_b^2(\mathbb{R}^p,\mathbb{R}^q)$. There exists a constant $C<+\infty$ such that the following inequality holds for every $(x,y,z,w)\in(\mathbb{R}^p)^4$:
    \begin{equation*}
        \|f(x)-f(y)-f(z)+f(w)\|\leq C\|x-y-z+w\| + C\|z-w\|(\|z-x\|+\|w-y\|).
    \end{equation*}
\end{lem}

\begin{lem}\label{lm:discrete_mv_rate_2_part_3}
    Assume \hyperref[assump:A1]{(A1)}. Let $k\in\mathbb{N}$. There exists a constant $C<+\infty$ such that for every $(l,M)\in\mathbb{N}\times\mathbb{N}$, $t\in\{0,\Delta_l,\dots,T\}$, and $i\in\{1,\dots,M\}$:
    \begin{equation*}
        \mathbb{E}\left[ \left\|Z^{l,M,i}_t-Z^{l-1,M,i}_t - Z^{l,i}_t+Z^{l-1,i}_t\right\|^k \right]\leq C\Delta_l^{k/2}/M^{k/2}.
    \end{equation*}
\end{lem}
\begin{proof}
Assume $k$ is even. The inequality is proven following the same strategy as the proof of Lemma \ref{lm:discrete_mv_rate_2_part_2}, we write following \eqref{eq:Z_M_difference} and \eqref{eq:L_L+1_decomposition}:
    \begin{equation}\label{eq:Z_L_M_difference}
    	\begin{split}
    		 &Z^{l,M,i}_{t}  - Z^{l-1,M,i}_{t} -  Z^{l,i}_{t}  + Z^{l-1,i}_{t}\\
    		=\;& Z^{l,M,i}_{t-\Delta_{l-1}}  - Z^{l-1,M,i}_{t-\Delta_{l-1}} - Z^{l,i}_{t-\Delta_{l-1}}  + Z^{l-1,i}_{t-\Delta_{l-1}}\\
    		+\;& \bigg(a\left(Z^{l,M,i}_{t-\Delta_{l-1}},\bar{\xi}_{1}(Z^{l,M,i}_{t-\Delta_{l-1}}, \mu^{l,M}_{t-\Delta_{l-1}})\right) - a\left(Z^{l-1,M,i}_{t-\Delta_{l-1}},\bar{\xi}_{1}(Z^{l-1,M,i}_{t-\Delta_{l-1}}, \mu^{l-1,M}_{t-\Delta_{l-1}})\right)\\
    		-\;& a\left(Z^{l,i}_{t-\Delta_{l-1}},\bar{\xi}_{1}(Z^{l,i}_{t-\Delta_{l-1}}, \mu^{l}_{t-\Delta_{l-1}})\right) + a\left(Z^{l-1,i}_{t-\Delta_{l-1}},\bar{\xi}_{1}(Z^{l-1,i}_{t-\Delta_{l-1}}, \mu^{l-1}_{t-\Delta_{l-1}})\right)\bigg)\Delta_{l-1} \\
    		+\;& \cdots
    	\end{split}	
    \end{equation}
where we wrote $\cdots$ for brevity to indicate that the rest of terms are similar to first term in \eqref{eq:Z_L_M_difference} and follows the same general form as \eqref{eq:Z_M_difference} and \eqref{eq:L_L+1_decomposition}. We will show how to bound the first term in \eqref{eq:Z_L_M_difference} since the rest of the terms, which were indicated by $\cdots$, are treated similarly. define the vector $R^{l,M} = (Z^{l,M,i}_{t-\Delta_{l-1}},\bar{\xi}_{1}(Z^{l,M,i}_{t-\Delta_{l-1}}, \mu^{l,M}_{t-\Delta_{l-1}}))$ and write $a(R^{l,M})$ for $a\left(Z^{l,M,i}_{t-\Delta_{l-1}},\bar{\xi}_{1}(Z^{l,M,i}_{t-\Delta_{l-1}}, \mu^{l,M}_{t-\Delta_{l-1}})\right)$, similarly define $R^{l-1,M}$, $R^{l}$, and $R^{l-1}$. Applying Lemma \ref{lm:technical_lm_2} to the function $a$ we have
    \begin{equation}\label{eq:a_Z_eq}
        \begin{split}
            \|a(R^{l,M}) - a(R^{l-1,M}) - a(R^{l}) + a(R^{l-1})\|
            \leq\;&C\|R^{l,M} - R^{l-1,M}\|(\|R^{l,M} - R^{l}\|+\|R^{l-1,M} - R^{l-1}\|) \\
            +\;& C\|R^{l,M} - R^{l-1,M} - R^{l} + R^{l-1}\|,
        \end{split}
    \end{equation}
    Using the assumption that $\xi_1$ is Liptchitz, Lemma \ref{lm:discrete_mv_rate_2_part_2}, and Marcinkiewicz-Zygmund in a manner similar to \eqref{eq:MZ_1} we have $\mathbb{E}[\|R^{l,M}-R^{l-1,M}\|^{2k}]\leq  C\Delta_l^{k}$ and $\mathbb{E}[\|R^{l,M}-R^{l}\|^{2k}]\leq C\frac{1}{M^{k}}$,
    thus using Cauchy-Schwarz and the $C^p$ inequality we have
    \begin{equation}\label{eq:R_1}
        \begin{split}
            \;&\mathbb{E}[\|R^{l,M} - R^{l-1,M}\|^k(\|R^{l,M} - R^{l}\|+\|R^{l-1,M} - R^{l-1}\|)^k]\\
            \leq\;&C\left(\mathbb{E}[\|R^{l,M} - R^{l,M}\|^{2k}](\mathbb{E}[\|R^{l,M} - R^{l}\|^{2k}]+\mathbb{E}[\|R^{l-1,M} - R^{l-1}\|^{2k}])\right)^{1/2}\\
            \leq\;& C\Delta_l^{k/2}/M^{k/2}.
        \end{split}
    \end{equation}
    Similarly to \eqref{eq:a_Z_2} the term $\|R^{l,M} - R^{l-1,M} - R^{l} + R^{l-1}\|$ in \eqref{eq:a_Z_eq} is bounded by
    \begin{equation}\label{eq:R_2}
        \begin{split}
            &\,\mathbb{E}[\|R^{l,M} - R^{l-1,M} - R^{l} + R^{l-1}\|^k]\\
            \leq\,& C\mathbb{E}[\|Z_{t-\Delta_{l-1}}^{l,M,i} - Z_{t-\Delta_{l-1}}^{l-1,M,i} - Z_{t-\Delta_{l-1}}^{l,i} + Z_{t-\Delta_{l-1}}^{l-1,i}\|^k]\\
            +\,& C\mathbb{E}\bigg[\bigg|\frac{1}{M}\sum_{j=1}^M\Big(\xi_{1}(Z^{l,M,i}_{t-\Delta_{l-1}}, Z^{l,M,j}_{t-\Delta_{l-1}}) - \xi_{1}(Z^{l-1,M,i}_{t-\Delta_{l-1}}, Z^{l-1,M,j}_{t-\Delta_{l-1}})
            -\xi_{1}(Z^{l,i}_{t-\Delta_{l-1}}, Z^{l,j}_{t-\Delta_{l-1}}) \\
            +\,& \xi_{1}(Z^{l,i}_{t-\Delta_{l-1}}, Z^{l,j}_{t-\Delta_{l-1}})\Big)\bigg|^k\bigg]
            + C\mathbb{E}\bigg[\bigg|\frac{1}{M}\sum_{j\not=i}\Big(\xi_{1}(Z^{l,i}_{t-\Delta_{l-1}}, Z^{l,j}_{t-\Delta_{l-1}}) - \xi_{1}(Z^{l-1,i}_{t-\Delta_{l-1}}, Z^{l-1,j}_{t-\Delta_{l-1}})\\
            -\,&\int \xi_{1}(Z^{l,i}_{t-\Delta_{l-1}}, y)d\mu^{l}_{t-\Delta_{l-1}}(y) + \int \xi_{1}(Z^{l-1,i}_{t-\Delta_{l-1}}, y)d\mu^{l-1}_{t-\Delta_{l-1}}(y) \Big)\bigg|^k\bigg]
            + C\frac{1}{M^k}\mathbb{E}\bigg[\bigg|\xi_{1}(Z^{l,i}_{t-\Delta_{l-1}}, Z^{l,i}_{t-\Delta_{l-1}})\\ -\,& \xi_{1}(Z^{l-1,i}_{t-\Delta_{l-1}}, Z^{l-1,i}_{t-\Delta_{l-1}})
            -\int \xi_{1}(Z^{l,i}_{t-\Delta_{l-1}}, y)d\mu^{l}_{t-\Delta_{l-1}}(y) + \int \xi_{1}(Z^{l-1,i}_{t-\Delta_{l-1}}, y)d\mu^{l-1}_{t-\Delta_{l-1}}(y)\bigg|^k \bigg].
        \end{split}
    \end{equation}
    Using Lemmata \ref{lm:technical_lm_2} and then \ref{lm:discrete_mv_rate_2_part_2} and with Cauchy-Schwarz similarly to \ref{eq:R_1}, the second expectation after the inequality in \eqref{eq:R_2} can be bounded from above by $C\sup_i \mathbb{E}[\|Z_{t-\Delta_{l-1}}^{l,M,i} - Z_{t-\Delta_{l-1}}^{l-1,M,i} - Z_{t-\Delta_{l-1}}^{l,i} + Z_{t-\Delta_{l-1}}^{l-1,i}\|^k] + C\Delta_L^{k/2}/M^{k/2}$. For the third expectations we use Fubini and Marcinkiewicz-Zygmund inequality in a manner similar to \eqref{eq:MZ_1}, $\xi_1$ is Lipchitz, and Lemma \ref{lm:discrete_mv_rate_2_part_2}:
    \begin{eqnarray*}
        \begin{split}
            \,&\mathbb{E}\bigg[\bigg|\sum_{j\not=i}\Big(\xi_{1}(Z^{l,i}_{t-\Delta_{l-1}}, Z^{l,j}_{t-\Delta_{l-1}}) - \xi_{1}(Z^{l-1,i}_{t-\Delta_{l-1}}, Z^{l-1,j}_{t-\Delta_{l-1}})
            -\int \xi_{1}(Z^{l,i}_{t-\Delta_{l-1}}, y)d\mu^{l}_{t-\Delta_{l-1}}(y)\\ +\,& \int \xi_{1}(Z^{l-1,i}_{t-\Delta_{l-1}}, y)d\mu^{l-1}_{t-\Delta_{l-1}}(y) \Big)\bigg|^k\bigg]\\
            \leq\,& C\mathbb{E}\bigg[\bigg|\sum_{j\not=i}\Big(\xi_{1}(Z^{l,i}_{t-\Delta_{l-1}}, Z^{l,j}_{t-\Delta_{l-1}}) - \xi_{1}(Z^{l-1,i}_{t-\Delta_{l-1}}, Z^{l-1,j}_{t-\Delta_{l-1}})
            -\mathbb{E}[\xi_{1}(Z^{l,i}_{t-\Delta_{l-1}}, Z^{l,j}_{t-\Delta_{l-1}})]\\ +\,& \mathbb{E}[\xi_{1}(Z^{l-1,i}_{t-\Delta_{l-1}}, Z^{l-1,j}_{t-\Delta_{l-1}})] \Big)^2\bigg|^{k/2}\bigg]\\
            \leq\,& CM^{k/2}\sup_i \mathbb{E}[\|Z^{l,i}_{t-\Delta_{l-1}}-Z^{l,i}_{t-\Delta_{l-1}}\|^k]\\
            \leq\,& CM^{k/2}\Delta_{l}^{k/2}.
        \end{split}
    \end{eqnarray*}
    Using $\xi_1$ is Lipchitz and Lemma \ref{lm:discrete_mv_rate_2_part_2} the last expectation in \eqref{eq:R_2} is bounded by $C\Delta_l^{k/2}/M^k$. Therefore we have shown
    \begin{eqnarray*}
        \begin{split}
    		\;& \mathbb{E}\bigg[\bigg\|a\left(Z^{l,M,i}_{t-\Delta_{l-1}},\bar{\xi}_{1}(Z^{l,M,i}_{t-\Delta_{l-1}}, \mu^{l,M}_{t-\Delta_{l-1}})\right) - a\left(Z^{l-1,M,i}_{t-\Delta_{l-1}},\bar{\xi}_{1}(Z^{l-1,M,i}_{t-\Delta_{l-1}}, \mu^{l-1,M}_{t-\Delta_{l-1}})\right)\\
    		-\;& a\left(Z^{l,i}_{t-\Delta_{l-1}},\bar{\xi}_{1}(Z^{l,i}_{t-\Delta_{l-1}}, \mu^{l}_{t-\Delta_{l-1}})\right) + a\left(Z^{l-1,i}_{t-\Delta_{l-1}},\bar{\xi}_{1}(Z^{l-1,i}_{t-\Delta_{l-1}}, \mu^{l-1}_{t-\Delta_{l-1}})\right)\bigg\|^k\bigg]\\
      \leq\;& C\sup_i\mathbb{E}\left[\|Z^{l,M,i}_{t-\Delta_{l-1}}  - Z^{l-1,M,i}_{t-\Delta_{l-1}} -  Z^{l}_{t-\Delta_{l-1}}  + Z^{l-1}_{t-\Delta_{l-1}}\|^k\right]
            + C\frac{\Delta_l^{k/2}}{M^{k/2}}.
        \end{split}
    \end{eqnarray*}
    A similar bound holds for the rest of terms in \eqref{eq:Z_L_M_difference}. Applying Lemma \ref{lm:technical_lm_3} several times we have
    \begin{equation*}
        \begin{split}
            \;&\sup_i\mathbb{E}\left[\|Z^{l,M,i}_{t}  - Z^{l-1,M,i}_{t} -  Z^{l}_{t}  + Z^{l-1}_{t}\|^k\right]\\\leq\;& (1+C\Delta_{l})\sup_i\mathbb{E}\left[\|Z^{l,M,i}_{t-\Delta_{l-1}}  - Z^{l-1,M,i}_{t-\Delta_{l-1}} -  Z^{l}_{t-\Delta_{l-1}}  + Z^{l-1}_{t-\Delta_{l-1}}\|^k\right]
            + C\frac{\Delta_l^{k/2}}{M^{k/2}}.
        \end{split}
    \end{equation*}
    Finally applying Lemma \ref{lm:technical_lm_1} to the sequence $\left\{\sup_i\mathbb{E}\left[\|Z^{l,M,i}_{s\Delta_{l-1}}  - Z^{l-1,M,i}_{s\Delta_{l-1}} -  Z^{l,i}_{s\Delta_{l-1}}  + Z^{l-1,i}_{s\Delta_{l-1}}\|^k\right]\right\}_{s=0}^{T\Delta_{l-1}^{-1}}$ finishes the proof.    
\end{proof}
\subsection{Bounds on $\tilde{Z}^{l,M}_t$ and $\pi^{l,M,N}_t$ increments}
\label{appen:pi_increments_bounds}
\begin{lem}\label{lm:Z_rate}
Assume \hyperref[assump:A1]{(A1)}. Let $k\in\mathbb{N}$. There exists a constant $C<+\infty$ such that for every $(l,M)\in\mathbb{N}\times\mathbb{N}$ and $t\in\{0,\Delta_l,\dots,T\}$ the following inequalities hold
    \begin{equation*}
        \max\{\mathbb{E}[\|\tilde{Z}^{l,M}_t\|^k], \mathbb{E}[\|\tilde{Z}^{l}_t\|^k]\}\leq C,
    \end{equation*}
	\begin{equation*}
        \mathbb{E}\left[\|\tilde{Z}^{l,M}_t - \tilde{Z}^{l}_t\|^k\right]\leq C\frac{1}{M^{k/2}},
	\end{equation*} 
    \begin{equation*}
        \mathbb{E}\left[\|\tilde{Z}^{l,M}_{t} - \tilde{Z}^{l-1,M}_{t}\|^k\right]\leq C\Delta_L^{k/2} ,
    \end{equation*}
	\begin{equation*}
        \mathbb{E}\left[\|\tilde{Z}^{l,M}_{t} - \tilde{Z}^{l-1,M}_{t} - \tilde{Z}^{l}_{t} + \tilde{Z}^{l-1}_{t}\|^k\right]\leq C\frac{\Delta_l^{k/2}}{M^{k/2}}.
	\end{equation*}
\end{lem}
\begin{proof}
    We assume $k$ is even. We will prove the second inequality only since the same technique we will use to prove it and the techniques used for proving Lemmata \ref{lm:discrete_mv_rate_2_part_2} and \ref{lm:discrete_mv_rate_2_part_3} can directly be applied to prove the other inequalities in this lemma statement. We write
	\begin{equation*}
		\begin{split}
			\tilde{Z}^{l,M}_t - \tilde{Z}^{l}_t =\;& \tilde{Z}^{l,M}_{t-\Delta_l} - \tilde{Z}^{l}_{t-\Delta_l} \\+\;& (a(\tilde{Z}^{l,M}_{t-\Delta_l},\bar{\xi}_1(\tilde{Z}^{l,M}_{t-\Delta_l},\mu^{l,M}_{t-\Delta_l}))-a(\tilde{Z}^{l}_{t-\Delta_l},\bar{\xi}_1(\tilde{Z}^{l}_{t-\Delta_l},\mu^{l}_{t-\Delta_l})))\Delta_l \\ +\; & (b(\tilde{Z}^{l,M}_{t-\Delta_l},\bar{\xi}_2(\tilde{Z}^{l,M}_{t-\Delta_l},\mu^{l,M}_{t-\Delta_l}))-b(\tilde{Z}^{l}_{t-\Delta_l},\bar{\xi}_2(\tilde{Z}^{l}_{t-\Delta_l},\mu^{l}_{t-\Delta_l})))\Delta W^l_t.
		\end{split}
	\end{equation*}
    The functions $a$ and $\xi_1$ are Lipschitz from Assumption \hyperref[assump:A1]{(A1)}, hence for a constant $C$ we have
\begin{equation}\label{eq:bound_2}
	\begin{split}
		\;&\mathbb{E}[\|a(\tilde{Z}^{l,M}_{t-\Delta_l},\bar{\xi}_1(\tilde{Z}^{l,M}_{t-\Delta_l},\mu^{l,M}_{t-\Delta_l}))-a(\tilde{Z}^{l}_{t-\Delta_l},\bar{\xi}_1(\tilde{Z}^{l}_{t-\Delta_l},\mu^{l}_{t-\Delta_l}))\|^k]\\
		\leq \;& C\bigg(\mathbb{E}[\|\tilde{Z}^{l,M}_{t-\Delta_l} - \tilde{Z}^{l}_{t-\Delta_l}\|^k] + \mathbb{E}\left[\left|\frac{1}{M}\sum_{i=1}^M\xi_1(\tilde{Z}^{l,M}_{t-\Delta_l},Z^{l,M,i}_{t-\Delta_l})-\frac{1}{M}\sum_{i=1}^M\xi_1(\tilde{Z}^{l}_{t-\Delta_l},Z^{l,i}_{t-\Delta_l})\right|^k\right]\\ 
        +\;& \mathbb{E}\left[\left|\frac{1}{M}\sum_{i=1}^M\xi_1(\tilde{Z}^{l}_{t-\Delta_l},Z^{l,i}_{t-\Delta_l})-\int \xi_1(\tilde{Z}^{l}_{t-\Delta_l},y)d\mu^l_{t-\Delta_l}(y)\right|^k\right].\\
  \end{split}
\end{equation}
The second expectation after the inequality in \eqref{eq:bound_2} is bounded by $C(\sup_i\mathbb{E}[\|\tilde{Z}^{l,M}_{t-\Delta_l} - \tilde{Z}^{l}_{t-\Delta_l}\|^k] + \mathbb{E}[\|Z^{l,M,i}_{t-\Delta_l} - Z^{l,i}_{t-\Delta_l}\|^k])$
and from Lemma \ref{lm:discrete_mv_rate_2_part_2} we have $\sup_i \mathbb{E}[\|Z^{l,M,i}_{t-\Delta_l} - Z^{l,i}_{t-\Delta_l}\|]\leq C/M^{k/2}$. The third expectation after the inequality in \eqref{eq:bound_2} is bounded above by $C/M^{k/2}$ using Fubini and Marcinkiewicz-Zygmund inequality in a manner similar to \eqref{eq:MZ_1}. Therefore we have
\begin{equation*}
	\begin{split}
		\;&\mathbb{E}\left[\|a(\tilde{Z}^{l,M}_{t-\Delta_l},\bar{\xi}_1(\tilde{Z}^{l,M}_{t-\Delta_l},\mu^{l,M}_{t-\Delta_l}))-a(\tilde{Z}^{l}_{t-\Delta_l},\bar{\xi}_1(\tilde{Z}^{l}_{t-\Delta_l},\mu^{l}_{t-\Delta_l}))\|^k\right]
		\leq  C\left(\mathbb{E}\left[\|\tilde{Z}^{l,M}_{t-\Delta_l} - \tilde{Z}^{l}_{t-\Delta_l}\|^k\right] + \frac{1}{M^{k/2}}\right).
	\end{split}
\end{equation*}
The rest of the proof is identical to the proof of Lemma \ref{lm:discrete_mv_rate_2_part_2}.
\end{proof}

\begin{lem}\label{lm:eta_ml}
	Assume \hyperref[assump:A1]{(A1-2)}. Let $\varphi\in\mathcal{C}^2_b(\mathbb{R}^d)\cap\mathcal{B}_b(\mathbb{R}^d)$, $t\in\{0,1,2,\dots,T\}$, and $k\in\mathbb{N}$. There exists $C<+\infty$ such that for every $(l,M)\in\mathbb{N}\times \mathbb{N}$:
	% \begin{equation*}
	% 	\mathbb{E}\left[\left|(\gamma_{t}^{l,M} - \gamma_{t}^{l-1,M})(\varphi)\right|^k\right] \leq C \Delta_l^{k/2}|\varphi|_{1}^{k},
	% \end{equation*}
	% \begin{equation*}
	% 	\mathbb{E}\left[\left|(\gamma_{t}^{l,M} - \gamma_{t}^{l})(\varphi)\right|^k\right] \leq C \frac{1}{M^{k/2}}|\varphi|_{1}^{k},
	% \end{equation*}
 	\begin{equation*}
		\mathbb{E}\left[\left|(\gamma_{t}^{l,M} - \gamma_{t}^{l-1,M})(\varphi)\right|^k\right] \leq C \Delta_l^{k/2}|\varphi|_{1}^{k},\qquad \mathbb{E}\left[\left|(\gamma_{t}^{l,M} - \gamma_{t}^{l})(\varphi)\right|^k\right] \leq C \frac{1}{M^{k/2}}|\varphi|_{1}^{k},
	\end{equation*}

	\begin{equation*}
		\mathbb{E}\left[\left|(\gamma_{t}^{l,M} - \gamma_{t}^{l-1,M} - \gamma_{t}^{l} + \gamma_{t}^{l-1})(\varphi)\right|^k\right] \leq C \frac{\Delta_l^{k/2}}{M^{k/2}}|\varphi|_{2}^{k}.
	\end{equation*}
\end{lem}
\begin{proof}
    Let $t\in\{0,1,\dots,T\}$. Define the function $\Phi:(\mathbb{R}^{d})^t\rightarrow\mathbb{R}$ by $\Phi(x_1,\dots,x_t) = \varphi(x_t)\prod_{i=1}^t G(x_i,y_i)$. From Assumption \hyperref[assump:A2]{(A2)} the function $\Phi$ is Lipchitz and $|\Phi|_1\leq C|\varphi|_1$. Using Jensen's inequality, the mean value theorem, and Lemma \ref{lm:Z_rate} we have
    \begin{equation*}
        \begin{split}
            \mathbb{E}\left[|(\gamma_t^{l,M} - \gamma^{l-1,M}_t)(\varphi)|^k\right] =\;& \mathbb{E}\left[\left|\mathbb{E}\left[ \Phi(\tilde{Z}_1^{l,M},\dots,\tilde{Z}_t^{l,M}) - \Phi(\tilde{Z}_1^{l-1,M},\dots,\tilde{Z}_t^{l-1,M})\bigg| \mathcal{F}^{l,M}_t,\mathcal{F}^{l-1,M}_t \right]\right|^k \right]\\
            \leq\;& \mathbb{E}\left[ \left|\Phi(\tilde{Z}_1^{l,M},\dots,\tilde{Z}_t^{l,M}) - \Phi(\tilde{Z}_1^{l-1,M},\dots,\tilde{Z}_t^{l-1,M})\right|^k \right]\\
            \leq\;& C\mathbb{E}\left[ |\varphi|_1^k \sum_{i=1}^t\|\tilde{Z}_i^{l,M} - \tilde{Z}_i^{l-1,M}\|^k \right]\\
            \leq\;& C|\varphi|_1^k \Delta_l^{k/2}.
        \end{split}
    \end{equation*}
    The second inequality in the lemma statement follows analogously. The proof of the third inequality is essentially the same but we use Lemma \ref{lm:technical_lm_2} and \ref{lm:Z_rate}.
\end{proof}
\begin{lem}\label{lm:pi_ml}
	Assume \hyperref[assump:A1]{(A1-2)}. Let $\varphi\in\mathcal{C}^2_b(\mathbb{R}^d)\cap\mathcal{B}_b(\mathbb{R}^d)$, $t\in\{0,1,2,\dots,T\}$, and $k\in\mathbb{N}$. There exists $C<+\infty$ such that for every $(l,M)\in\mathbb{N}\times \mathbb{N}$:
	% \begin{equation*}
	% 	\mathbb{E}\left[\left|(\pi_{t}^{l,M} - \pi_{t}^{l-1,M})(\varphi)\right|^k\right] \leq C \Delta_l^{k/2}|\varphi|_{1}^{k},
	% \end{equation*}
	% \begin{equation*}
	% 	\mathbb{E}\left[\left|(\pi_{t}^{l,M} - \pi_{t}^{l})(\varphi)\right|^k\right] \leq C \frac{1}{M^{k/2}}|\varphi|_{1}^{k},
	% \end{equation*}
 	\begin{equation*}
		\mathbb{E}\left[\left|(\pi_{t}^{l,M} - \pi_{t}^{l-1,M})(\varphi)\right|^k\right] \leq C \Delta_l^{k/2}|\varphi|_{1}^{k},\qquad \mathbb{E}\left[\left|(\pi_{t}^{l,M} - \pi_{t}^{l})(\varphi)\right|^k\right] \leq C \frac{1}{M^{k/2}}|\varphi|_{1}^{k},
	\end{equation*}

	\begin{equation*}
		\mathbb{E}\left[\left|(\pi_{t}^{l,M} - \pi_{t}^{l-1,M} - \pi_{t}^{l} + \pi_{t}^{l-1})(\varphi)\right|^k\right] \leq C \frac{\Delta_l^{k/2}}{M^{k/2}}|\varphi|_{2}^{k}.
	\end{equation*}
\end{lem}
\begin{proof}
The first two inequalities follow directly form Lemma \ref{lm:eta_ml}. The third inequality follows directly as well from Lemma \ref{lm:eta_ml} and Cauchy Schwarz inequality and using the algebraic identity
% \begin{eqnarray*}
%     \frac{a}{A} - \frac{b}{B} = \frac{1}{A}(a-b) - \frac{b}{AB}(A-B),
% \end{eqnarray*}
\begin{equation*}
\begin{split}
    \frac{a}{A} - \frac{b}{B} - \frac{c}{C} + \frac{d}{D} =\,& \frac{1}{A}(a-b-c+d) -\frac{b}{AB}(A-B-C+D) - \frac{1}{AC}(A-C)(c-d)\\
    -\,&\frac{1}{AB}(C-D)(b-d)+\frac{d}{CBD}(B-D)(C-D) + \frac{d}{ACB}(A-C)(C-D). \\
\end{split}
\end{equation*}
\end{proof}

\subsection{Proof of Theorem \ref{thm:single_level_error}}
\label{appen:proof_sl}
We begin this section by bounding the bias of the estimator $\pi^{L}(\varphi)$. We use Proposition 3.1 from \cite{koh} which we restate as Lemma \ref{lm:weak_error} below.
\begin{rem}
    Consider the process $Z_t' =Z_t-x_0$ and define the functions $a'(u,v)=a(u+x_0,v)$, $b'(u,v)=b(u+x_0,v)$, $\xi_j'(u,v)=\xi_j(u+x_0,v+x_0)$ for $j\in\{1,2\}$. $Z_t'$ satisfies SDE \eqref{eq:sde} but with $a,b,\xi_1,\xi_2,x_0$ replaced with $a',b',\xi_1',\xi_2',0$. $C$ in the Proposition 3.1 in \cite{koh} can be bounded by a constant that depends on the bounds of $a',b',\xi_1',\xi_2'$ and their derivatives which are finite by Assumption \hyperref[assump:A1]{A1}, this allows us to remove the dependence of $x_0$ from the constant $C$ in Proposition 3.1 of \cite{koh}.
\end{rem}
\begin{lem}\label{lm:weak_error}
	Assume \hyperref[assump:W2]{(W2)}. Let $t\in\{0,1,\dots,T\}$. For every $\varphi\in\mathcal{C}^{\infty}_b(\mathbb{R}^d)\cap\mathcal{B}_b(\mathbb{R}^d)$ there exists $C<+\infty$ such that for every $l\in\mathbb{N}_0$:
     \begin{equation*}
         \left|P_{\mu_{t-1},t}(\varphi) - P_{\mu_{t-1}^l,t}(\varphi)\right|_0\leq C\Delta_l.
     \end{equation*}
\end{lem}
\begin{lem}\label{lm:pi_eta_estimates_1_part_1}
Assume \hyperref[assump:A2]{(A2)}. Let $t\in\{0,\dots,T\}$ There exists a constant $C<+\infty$ such that for every $l\in\mathbb{N}_0$ and  $\varphi\in\mathcal{B}_b(\mathbb{R}^d)$:
    \begin{equation*}
        \max(|\gamma_t(\varphi)|,|\gamma_t^l(\varphi)|) \leq C|\varphi|_0,
    \end{equation*}
\end{lem}
\begin{proof}
    Follows directly from Assumption \hyperref[assump:A2]{(A2)} and the definitions of $\gamma_t(\varphi)$ and $\gamma_t^l(\varphi)$.
\end{proof}
\begin{lem}\label{lm:pi_eta_estimates_1_part_2}
Assume \hyperref[assump:A1]{(A1-2)} and \hyperref[assump:W1]{(W1-2)}. Let $t\in\{0,\dots,T\}$. For every $\varphi\in\mathcal{C}^{\infty}_b(\mathbb{R}^d)\cap\mathcal{B}_b(\mathbb{R}^d)$ there exists a constant $C<+\infty$ such that for every $(l,M)\in\mathbb{N}_0\times\mathbb{N}$
    \begin{equation*}
        \left|\gamma_t(\varphi) - \gamma_t^{l}(\varphi)\right|_0 \leq C\Delta_l,
    \end{equation*}
    \begin{equation*}
        \left|\pi_t(\varphi) - \pi_t^{l}(\varphi)\right|_0 \leq C\Delta_l.
    \end{equation*}
\end{lem}
\begin{proof}
    For concise notation we define the functions $G_t(x)=G(x,y_t)$ for every $t\in\{0,\dots,T\}$ and $x\in\mathbb{R}^d$. We will prove the inequalities using induction.
    The first inequality is satisfied trivially at $t=0$ since $\gamma_0=\gamma_0^L=\delta_{x_0}$. Assume the inequality holds for $t-1$ where $t\geq 1$, we write
    \begin{equation*}
    \begin{split}
        \gamma_t(\varphi) - \gamma_{t}^l(\varphi)
        =\gamma_{t-1}^l(P_{\mu_{t-1},t}(G_t\varphi) - P^l_{\mu^l_{t-1},t}(G_t\varphi)) + (\gamma_{t-1}-\gamma_{t-1}^l)(P_{\mu_{t-1},t}(G_t\varphi)).
    \end{split}
    \end{equation*}
    Lemma \ref{lm:pi_eta_estimates_1_part_1} and Lemma \ref{lm:weak_error} yields the bound for the first term. Assumption \hyperref[assump:W1]{(W1)} allows us to apply the induction hypothesis to bound the second term. %Thus
    %\begin{equation*}
    %    |\gamma_t(\varphi) - \gamma_t^l(\varphi)|\leq C|(P_{\mu_{t-1},t}-P^l_{\mu^l_{t-1},t})(G_t\varphi)|_0 + C\Delta_l\leq C\Delta_l.
    %\end{equation*}
    The second inequality follows from the first one and the equality
    \begin{equation*}
        \begin{split}
            \;&\pi_t(\varphi) - \pi_t^{l}(\varphi)=
            \frac{1}{\gamma_t(1)\gamma_t^l(1)}(\gamma_t^l(1)-\gamma_t(1))\gamma_t^l(\varphi) - \frac{1}{\gamma_t(1)}(\gamma^l_t(\varphi) - \gamma_t(\varphi)).
        \end{split}
    \end{equation*}
    Notice that the quantities $\gamma_t(1)$ and $\gamma^l_t(1)$ are bounded below away from $0$ because of Assumption \hyperref[assump:A2]{(A2)}.
\end{proof}
The following lemma is a consequence of the standard theory of $\mathbb{L}_p$-bounds for particle filters (Theorem 2.9 in \cite{moral_2000}). %Although the theorem is stated for the predictors $\eta$ the Assumption \hyperref[assump:A2]{(A2)} easily allows us to use the Theorem for the filters $\pi$. 
%Conditioning on $\sigma(\mathcal{F}^{l-1,M}_t,\mathcal{F}^{l,M}_t)$, following Theorem 2.9 in \cite{moral_2000}, and employing Assumption \hyperref[assump:A2]{(A2)} we can bound $\mathbb{E}[ (\gamma^{l,M}_t(\varphi)- \gamma^{l,M,N}_t(\varphi))^2 |\mathcal{F}^{l,M}_t]\leq C\frac{|G\varphi|_0}{N}$
%\begin{rem}
%We apply the bound in \cite{moral_2000} to the conditioned mean squared error of the predictors $\mathbb{E}[ (\gamma^{l,M}_t(\varphi)- \gamma^{l,M,N}_t(\varphi))^2 |\mathcal{F}^{l,M}_t]$, since the function $G$ is bounded, the constant in the bound in \cite{moral_2000} can be bounded from above by a deterministic constant independent of $\mathcal{F}_t^{l,M}$. This gives us a bound on $\mathbb{E}[ (\gamma^{l,M}_t(\varphi)- \gamma^{l,M,N}_t(\varphi))^2]$. Again, since the function $G$ is bounded the bound in \cite{moral_2000} holds for the filters $\pi_t^{l,M,N}$ and $\pi_t^{l,M}$.
%\end{rem} 
\begin{lem}\label{lm:pf1}
	Assume \hyperref[assump:A1]{(A1-2)} and let $t\in\{1,\dots,T\}$. There exists $C<+\infty$ such that for every $\varphi\in\mathcal{B}_p(\mathbb{R}^d)$ and $(l,M,N)\in\mathbb{N}_0\times \mathbb{N}\times \mathbb{N}$:
	\begin{equation*}
		\mathbb{E}\left[ \left|\pi^{l,M}_t(\varphi)- \pi^{l,M,N}_t(\varphi) \right|^2 \right] \leq C\frac{1}{N}|\varphi|_0^2.
	\end{equation*}
\end{lem}
\begin{rem}\label{rm:C_deterministic_1}
    Because $G$ is bounded we can apply Theorem 2.9 in \cite{moral_2000} to obtain the bound $\mathbb{E}[ (\pi^{l,M}_t(\varphi)- \pi^{l,M,N}_t(\varphi))^2 |\mathcal{F}^{l,M}_t]\leq C\frac{|\varphi|_0}{N}$ where $C$ is $\mathcal{F}^{l,M}_t$. However, following the formula for $C$ in Theorem 2.9 in \cite{moral_2000} and Assumption \hyperref[assump:A2]{(A2)} allow us to bound $C$ from above by a deterministic constant. This justifies the deterministic $C$ in Lemma \ref{lm:pf1}.
\end{rem}

\begin{proof}[\bf{Proof of Theorem \ref{thm:single_level_error}}:]
Using the $C^p$ inequality
    \begin{equation*}
    \begin{split}
        &\mathbb{E}\left[(\pi_t(\varphi) - \pi_t^{L,M,N}(\varphi))^2\right]\\
        \leq\,& C\left( (\pi_t(\varphi) - \pi_t^{L}(\varphi))^2+ \mathbb{E}\left[(\pi_t^{L}(\varphi) - \pi_t^{L,M}(\varphi))^2\right] + \mathbb{E}\left[(\pi_t^{L,M}(\varphi) - \pi_t^{L,M,N}(\varphi))^2\right] \right).
        \end{split}
    \end{equation*}
    Lemmata \ref{lm:pi_eta_estimates_1_part_2}, \ref{lm:pi_ml}, and \ref{lm:pf1} finish the proof.
\end{proof}
\subsection{Proof of Theorem \ref{thm:ml_error}}
\label{appen:proof_ml}
We condition on $\sigma(\mathcal{F}^{l-1,M}_t,\mathcal{F}^{l,M}_t)$ and consider Theorem C.1 and Lemma C.3 in \cite{mlpf} under this conditioning. Following the same reasoning as in Remark \ref{rm:C_deterministic_1} the constant $C(n,\varphi)$ in \cite{mlpf} can be assumed to be deterministic in our setting because the functions $a,b,\xi_1,\xi_2,G$ are bounded by Assumptions \hyperref[assump:A1]{(A1-2)}. Thus we have that for $s\in\{1,2\}$:
\begin{equation*}
		\left|\mathbb{E}\left[ \left( \pi^{l,M,N}_{t}(\varphi) - \pi^{l-1,M,N}_{t}(\varphi) - \pi^{l,M}_{t}(\varphi) + \pi^{l-1,M}_{t}(\varphi)\right)^s\right]\right| \leq C\frac{\mathbb{E}[B(n)]^{s/2}}{N},
\end{equation*}
where we used the inequality $\mathbb{E}[\sqrt{B(n)}]\leq \mathbb{E}[B(n)]^{1/2}$ and $B(n)$ is defined in equation (C.1) in \cite{mlpf}. Lemma \ref{lm:Z_rate} verifies an in expectation (over all randomness) version of Assumption D.1 in \cite{mlpf} from which a version of Theorem D.5 in \cite{mlpf} follows (with expectation taken over all randomness). We apply the $C^p$ inequality on $B(n)$, the first sum corresponding to maximum coupling is bounded by our version of Theorem D.5 in \cite{mlpf} and for the other terms Lemmata \ref{lm:Z_rate} and \ref{lm:eta_ml} suffice. Therefore we have $\mathbb{E}[B(n)]\leq \Delta^{1/2}$. We state the result of this discussion as the following Lemma.

\begin{lem}\label{lm:mlpf_increment}
	Assume \hyperref[assump:A1]{(A1-2)}. Let $\varphi\in\mathcal{C}_b^1(\mathbb{R}^d)\cap\mathcal{B}_b(\mathbb{R}^d)$ and $t\in\{1,\dots,T\}$. There exists $C<+\infty$ such that for every $(l,M,N)\in\mathbb{N}\times\mathbb{N}\times \mathbb{N}$ and $s\in\{1,2\}$:
	\begin{equation*}
		\left|\mathbb{E}\left[ \left( \pi^{l,M,N}_{t}(\varphi) - \pi^{l-1,M,N}_{t}(\varphi) - \pi^{l,M}_{t}(\varphi) + \pi^{l-1,M}_{t}(\varphi)\right)^s\right]\right| \leq C\frac{\Delta_l^{s/4}}{N}
	\end{equation*}
\end{lem}

\begin{proof}[\bf{Proof of Theorem \ref{thm:ml_error}:}]
    \begin{equation*}
        \begin{split}
            \;& \mathbb{E}\left[ (\pi_t(\varphi)-\widehat{\pi^L_t(\varphi)})^2 \right]\\
            =\;& \mathbb{E}\bigg[ \Big((\pi_t(\varphi) - \pi_t^L(\varphi)) + (\pi_t^{0}(\varphi) - \pi_t^{0,M_0}(\varphi)) + ( \pi_t^{0,M_0}(\varphi) - \pi_t^{0,M_0,N_0}(\varphi)) \\
            +\;& \sum_{l=1}^L(\pi^{l}_t - \pi^{l-1}_t-\pi^{l,M_l}_t +\pi^{l-1,M_l}_t)(\varphi)  + \sum_{l=1}^L(\pi^{l,M_l}_t - \pi^{l-1,M_l}_t-\pi^{l,M_l,N_l}_t +\pi^{l-1,M_l,N_l}_t)(\varphi)\Big)^2 \bigg]\\
            \leq \;&C(\pi_t(\varphi)-\pi_t^L(\varphi))^2 + C\mathbb{E}\left[ (\pi_t^{0}(\varphi) - \pi_t^{0,M_0}(\varphi)) ^2 \right] + C\mathbb{E}\left[  \left( \sum_{l=1}^L(\pi^{l}_t - \pi^{l-1}_t-\pi^{l,M_l}_t +\pi^{l-1,M_l}_t)(\varphi)\right)^2\right]\\
            +\;& C\mathbb{E}\left[ ( \pi_t^{0,M_0}(\varphi) - \pi_t^{0,M_0,N_0}(\varphi))^2\right] + C\mathbb{E}\left[\left(\sum_{l=1}^L(\pi^{l,M_l}_t - \pi^{l-1,M_l}_t-\pi^{l,M_l,N_l}_t +\pi^{l-1,M_l,N_l}_t)(\varphi)\right)^2\right].
        \end{split}
    \end{equation*}
    The first term is bounded using $C\Delta_L^2$ by Lemma \ref{lm:pi_eta_estimates_1_part_2}, the bounds for the second and third terms follow from $C^p$ inequality and Lemma \ref{lm:pi_ml}. The fourth is the error of a standard particle filter stated in Lemma \ref{lm:pf1}. For the fifth term we use Lemma \ref{lm:mlpf_increment} as follows
    
    \begin{equation*}
        \begin{split}
            \;& \mathbb{E}\left[\left(  \sum_{l=1}^L(\pi_{t}^{l,M_l} - \pi_{t}^{l-1,M_l}-\pi_{t}^{l,M_l,N_l} +\pi_{t}^{l-1,M_l,N_l})(\varphi)   \right)^2\right]\\
            =\;& \sum_{l=1}^L\mathbb{E}\left[\left(  (\pi_{t}^{l,M_l} - \pi_{t}^{l,M_l}-\pi_{t}^{l-1,M_l,N_l} +\pi_{t}^{l-1,M_l,N_l})(\varphi)    \right)^2\right]
            +\sum_{p\not=q} \\\;& \mathbb{E}\left[ (\pi_{t}^{p,M_{p}} - \pi_{t}^{p-1,M_{p}}-\pi_{t}^{p,M_{p},N_{p}} +\pi_{t}^{p-1,M_{p},N_{p}})(\varphi)  \right]\mathbb{E}\left[ (\pi_{t}^{q,M_{q}} - \pi_{t}^{q-1,M_{q}}-\pi_{t}^{q,M_{q},N_{q}} +\pi_{t}^{q-1,M_{q},N_{q}})(\varphi) \right]\\
            \leq\;& C\sum_{l=1}^L\frac{\Delta_l^{1/2}}{N_l} + C\sum_{p\not=q}\frac{\Delta_p^{1/4}\Delta_q^{1/4}}{N_pN_q}\\
            \leq\;& C\sum_{l=1}^L\frac{\Delta_l^{1/2}}{N_l},
        \end{split}
    \end{equation*}
	where for the last line first notice that Assumtion \hyperref[assump:A3]{(A3)} implies that $N_l\geq (L-l)^2/2$ for $0\leq l\leq L-1$ then using Cauchy Schwarz inequality
	\begin{equation*}
		 C\sum_{p\not=q}\frac{\Delta_p^{1/4}\Delta_q^{1/4}}{N_pN_q}\leq C\left(\sum_{l=1}^L\frac{\Delta_l^{1/4}}{N_l}\right)^2\leq C\left(\sum_{l=1}^L\frac{1}{N_l}\right) \left(\sum_{l=1}^L\frac{\Delta_l^{1/2}}{N_l}\right)< C\left(1+\frac{1}{2}\sum_{l=1}^{\infty}\frac{1}{l^2}\right) \left(\sum_{l=1}^L\frac{\Delta_l^{1/2}}{N_l}\right).
	\end{equation*}
\end{proof}

\end{document}